\documentclass[11pt]{amsart}
\usepackage{graphicx}
\usepackage{amscd}
\usepackage{amsmath}
\usepackage{amsxtra}
\usepackage{amsfonts}
\usepackage{amssymb}
\usepackage{xcolor}

\newtheorem{theorem}{Theorem}[section]
\newtheorem{corollary}[theorem]{Corollary}
\newtheorem{lemma}[theorem]{Lemma}
\newtheorem{proposition}[theorem]{Proposition}
\newtheorem{assumptions}[theorem]{Assumptions}

\theoremstyle{definition}
\newtheorem{definition}[theorem]{Definition}

\newtheorem{remark}[theorem]{Remark}

\newtheorem{example}[theorem]{Example}
\theoremstyle{remark}
\newtheorem{claim}[theorem]{Claim}
\renewcommand{\theclaim}{\textup{\theclaim}}

\newtheorem*{acknowledgements}{Acknowledgements}

\numberwithin{equation}{section}

\newcommand\mtiny[1]{\mbox{\tiny\ensuremath{#1}}}

\def\openone%{\hbox{\upshape \small1\kern-3.3pt\normalsize1}}

{\mathchoice

{\hbox{\upshape \small1\kern-3.3pt\normalsize1}}

{\hbox{\upshape \small1\kern-3.3pt\normalsize1}}

{\hbox{\upshape \tiny1\kern-2.3pt\SMALL1}}

{\hbox{\upshape \Tiny1\kern-2pt\tiny1}}}

\makeatletter

\newbox\ipbox

\newcommand{\ip}[2]{\left\langle #1\, , \,#2\right\rangle}
\newcommand{\diracb}[1]{\left\langle #1\mathrel{\mathchoice

{\setbox\ipbox=\hbox{$\displaystyle \left\langle\mathstrut
#1\right.$}

\vrule height\ht\ipbox width0.25pt depth\dp\ipbox}

{\setbox\ipbox=\hbox{$\textstyle \left\langle\mathstrut
#1\right.$}

\vrule height\ht\ipbox width0.25pt depth\dp\ipbox}

{\setbox\ipbox=\hbox{$\scriptstyle \left\langle\mathstrut
#1\right.$}

\vrule height\ht\ipbox width0.25pt depth\dp\ipbox}

{\setbox\ipbox=\hbox{$\scriptscriptstyle \left\langle\mathstrut
#1\right.$}

\vrule height\ht\ipbox width0.25pt depth\dp\ipbox}

}\right. }

\newcommand{\dirack}[1]{\left. \mathrel{\mathchoice

{\setbox\ipbox=\hbox{$\displaystyle \left.\mathstrut
#1\right\rangle$}

\vrule height\ht\ipbox width0.25pt depth\dp\ipbox}

{\setbox\ipbox=\hbox{$\textstyle \left.\mathstrut
#1\right\rangle$}

\vrule height\ht\ipbox width0.25pt depth\dp\ipbox}

{\setbox\ipbox=\hbox{$\scriptstyle \left.\mathstrut
#1\right\rangle$}

\vrule height\ht\ipbox width0.25pt depth\dp\ipbox}

{\setbox\ipbox=\hbox{$\scriptscriptstyle \left.\mathstrut
#1\right\rangle$}

\vrule height\ht\ipbox width0.25pt depth\dp\ipbox}

} #1\right\rangle}

\newcommand{\beq}{\begin{equation}}

\newcommand{\eeq}{\end{equation}}

\newcommand{\cj}[1]{\overline{#1}}

\newcommand{\bz}{\mathbb{Z}}

\newcommand{\dist}{\textup{dist}}
\newcommand{\br}{\mathbb{R}}
\newcommand{\bc}{\mathbb{C}}

\newcommand{\bn}{\mathbb{N}}

\def\blfootnote{\xdef\@thefnmark{}\@footnotetext}

\newcommand{\Prob}{\operatorname*{Prob}}

\renewcommand{\mod}{\operatorname{mod}}

\input xy
\xyoption{all}
\usepackage{amssymb}

\newcommand{\Span}{\overline{\operatorname*{span}}}

\def\R{\mathbb{R}}

\def\-{^{-1}}

\def\ty{\emptyset}

\def\C{\mathbb{C}}
\def\Z{\mathbb{Z}}

\begin{document}

\title[Parseval Frames from Compressions of Cuntz Algebras]{Parseval Frames from Compressions of Cuntz Algebras}
\author{Nicholas Christoffersen }
\address{[Nicholas Christoffersen ] University of Colorado\\
    Department of Mathematics\\
    Campus Box 395\\
    2300 Colorado Avenue\\
    Boulder, CO 80309-0395\\} \email{nicholas.christoffersen@colorado.edu}

\author{Dorin Ervin Dutkay}
\address{[Dorin Ervin Dutkay] University of Central Florida\\
	Department of Mathematics\\
	4000 Central Florida Blvd.\\
	P.O. Box 161364\\
	Orlando, FL 32816-1364\\
U.S.A.\\} \email{Dorin.Dutkay@ucf.edu}
\author{Gabriel Picioroaga}
\address{[Gabriel Picioroaga] University of South Dakota\\
          Department of Mathematical Sciences\\
          414 E. Clark St. \\
          Vermillion, SD 57069\\
U.S.A. \\} \email{Gabriel.Picioroaga@usd.edu}
\author{Eric S. Weber}
\address{[Eric S. Weber] Iowa State University\\
          Department of Mathematis\\
396 Carver Hall\\
411 Morrill Road\\
Ames, IA 50011\\
U.S.A. \\} \email{esweber@iastate.edu}

\thanks{}

\begin{abstract}
A row co-isometry is a family $(V_i)_{i=0}^{N-1}$ of operators on a Hilbert space, subject to the relation 
$$\sum_{i=0}^{N-1}V_iV_i^*=I.$$
As shown in \cite{BJK00}, row co-isometries appear as compressions of representations of Cuntz algebras. 

In this paper we will present some general constructions of Parseval frames for Hilbert spaces, obtained by iterating the operators $V_i$ on a finite set of vectors. The constructions are based on random walks on finite graphs. As applications of our constructions we obtain Parseval Fourier bases on self-affine measures and Parseval Walsh bases on the interval. 
\end{abstract}
\maketitle \tableofcontents

\section{Introduction}

Structured bases appear in harmonic analysis, operator theory, and approximation theory, among other areas.  The classical example of a structure basis is an exponential (Fourier) basis, which gives rise to Fourier series expansions.  A probability measure $\mu$ on $\mathbb{R}^{d}$ is \emph{spectral} if there exists a sequence of exponential functions that form an orthonormal basis for $L^2(\mu)$.  Lebesgue measure on the unit (hyper-)cube is spectral; remarkably, Jorgensen and Pedersen initially showed that there are fractal measures which are spectral \cite{JoPe98}.  Wavelet bases \cite{Dau92} are another ubiquitous class of structured bases.  These arise from the action of a system of unitary operators on $L^2(\mathbb{R})$--dilations and translations \cite{DaLa98}--that encode natural operations on the latent space.  Wavelets, however, lead a double existence between $L^2(\mathbb{R})$ and $\ell^2(\mathbb{Z})$, as elucidated by Mallat \cite{Mal89}.  Wavelet bases in $\ell^2(\mathbb{Z})$ are generated by the iterated action of a finite number of (co-)isometries.  These co-isometries give rise to a notion of scale in $\ell^2(\mathbb{Z})$, and the corresponding scale decomposition is referred to as \emph{the cascade algorithm}.

In the case of wavelet bases, the co-isometries $\{ S_i^{*} \}$ satisfy what are now known as the Cuntz relations:
$$\sum_{i=0}^{N-1}S_iS_i^*=I,\quad S_j^*S_i=\delta_{j,i}I$$
These relations were in fact observed by engineers--albeit without a precise mathematical formulation--in the first half of the twentieth century.  J. Cuntz is credited with the discovery and thorough study of the algebras generated by such systems of co-isometries and the associate representation theory of those algebras~\cite{Cun77}.  The role played by the Cuntz algebras in wavelet theory was described in the work of Bratteli and Jorgensen \cite{BrJo02a,BrJo02b,BrJo00,BrJo97}. Orthonormal wavelet bases (ONB) are constructed from various choices of quadrature mirror filters (e.g. see \cite{Dau92}). These filters are in one-to-one correspondence with certain representations of a Cuntz algebra.
% Due to their encoding properties, the Cuntz relations were anticipated by engineers albeit without a precise mathematical formulation. J. Cuntz is credited with the discovery and thorough study of their representation theory~\cite{Cun77}. Practical implementation on finite dimensional approximations allows for perfect reconstruction of signals from previous subband decomposition. 

The Cuntz relations give an elegant way to understand the geometry of cascade algorithms (for example the Discrete Wavelet Transform). The first identity is used to decompose a vector $v$ into a ``cascade'' of bits  $S_{i}^*v$, where $i$ indexes from 0 to $ N-1$. Recovery of $v$ is obtained through the same identity: apply each $S_i$ to the bits and sum it all up.  The second relation (orthogonality of operators) tells us that bit interference/overlapping is avoided. 
%The Cuntz relations enter signal processing by way of $N$ isometries
%$S_i:H\rightarrow H$ realized on a Hilbert space $H$, which frequently is $L^2[0,1]$ or $l^2(\mathbb{Z})$, subject to relations:
%$$\sum_{i=0}^{N-1}S_iS_i^*=I,\quad S_j^*S_i=\delta_{j,i}I$$

In \cite{DPS14} new examples of orthonormal bases and also classic ones were found to be generated by Cuntz algebra representations. More precisely, a multitude of ONBs such as Fourier bases on fractals, Walsh bases on the unit interval, and piecewise exponential bases on the middle-third Cantor set can be gathered under the umbrella of Cuntz algebra representations.  Moreover, in \cite{DuPiSi19} a connection between the ONB property and irreducibility of a Cuntz algebra representation was made in the particular case of Walsh systems.  

However, orthonormal bases are sometimes too restrictive--for example, not all measures are spectral \cite{JoPe98}, and \emph{orthogonal} wavelets may lack certain desirable properties \cite{Dau92}.  Non-orthogonal expansions given by \emph{frames} were introduced by Duffin and Schaffer \cite{DuSc52} and popularized by \cite{DGM86}.  A \emph{Parseval frame} for a Hilbert space $H$ is a family of vectors $\{e_i\}_{i\in I}$ such that 
$$\|v\|^2=\sum_{i\in I}|\ip{v}{e_i}|^2,\quad (v\in H).$$
Parseval frames arise naturally as images of orthonormal bases under a co-isometry; notably, by the Naimark dilation theorem, all Parseval frames have this form \cite{HaLa00}.

One of our main motivations and applications comes from the harmonic analysis of fractal measures. The study of orthogonal Fourier series on fractal measures began with the paper \cite{JoPe98}, in connection with the Fuglede conjecture \cite{Fug74}. Jorgensen and Pedersen proved that, for the Cantor measure $\mu_4$ on the Cantor set $C_4$ with scale 4 and digits 0 and 2, the set of exponential functions 
$$\left\{ e^{2\pi i\lambda x}: \lambda=\sum_{k=0}^n 4^k l_l, n\in \bn, l_k\in\{0,1\}\right\},$$
is an orthonormal basis of $L^2(\mu_4)$.  Many more examples of spectral measures have been constructed since, see, e.g., \cite{Str00,DJ06, DHL19}. For the classical middle-third Cantor measure, Jorgensen and Pedersen proved that this construction is not possible \cite[Section 6]{JoPe98}, so that measure is not spectral.  Strichartz \cite{Str00} posed the natural question of whether this measure has a frame of exponential functions.  This question is still open!  Non-orthogonal (but non-frame) Fourier series expansions for the middle-third Cantor measure were constructed in \cite{herr2017fourier}.

Motivated by Strichartz's question, in \cite{PiWe17}, weighted Fourier Frames were obtained for the Cantor $C_4$ set by making the set bigger, then constructing a basis for the bigger set, and then projecting the basis onto the Cantor set. Using the same dilation technique, in \cite{DuRa18, DuRa20}, a multitude of Parseval frames of weighted exponential functions and generalized Walsh bases were constructed for self-affine measures and for the unit interval.  These constructions lead naturally to reconsidering the Cuntz relations, noting that only the first relation is needed to reconstruct a signal originally decomposed by a cascade algorithm. In such a set-up the operators $(S_i)_{i=0}^N$ are called row co-isometries.  

In this paper we continue with the philosophy first emphasized in \cite{HaLa00}--that frames are compressions of orthogonal bases--by considering compressions of Cuntz algebras.  Indeed, one key idea is a dilation result from \cite{BJK00} (Theorem \ref{dil} in the next section) which allows one to extend a representation of a row co-isometry (which lacks the orthogonality constraints) to a ``genuine'' Cuntz algebra representation.   We extend this philosophy to the pair {\it{Parseval frame, row co-isometries}}.   By doing so, we will present a general framework for the construction of Parseval frames and orthonormal bases from row co-isometries, a framework which includes both of the following settings: (i) Fourier bases on self-affine measures and (ii) Walsh bases on the interval.  We show that the new results are effective and can capture and unify previously obtained examples of Parseval frames and ONBs. 

In Section 2 we include some definitions and notations; Section 3 we state the main results; in Section 4, we present the proofs and some related results, and in Section 5 we apply the theory to various classes of examples.

\section{Preliminaries and notations}

\begin{definition}\label{p1}
Let $H$ be a Hilbert space. A family of $N$ bounded operators $\{S_i\}_{i=0}^{N-1}$ on $H$ that satisfy the relations
$$\sum_{i=0}^{N-1}S_iS_i^*=I_H,\quad S_j^*S_i=\delta_{j,i}I, \quad(i,j\in\{0,\dots,N-1\})$$
is called {\it a family of Cuntz isometries}, or a {\it a representation of the Cuntz algebra} $\mathcal O_N$. 

Note that the second relation implies that the operators are isometries with orthogonal ranges, and the first relation implies that the sum of the ranges add up to the whole space. 

Such a representation is called {\it irreducible} if the only operators $A$ on $H$, that commute with $S_i$ and $S_i^*$, i.e., $AS_i=S_iA$, $AS_i^*=S_i^*A$, for all $i\in\{0,\dots,N-1\}$, are multiples of the identity $A=cI_H$, for some $c\in\bc$. Equivalently, a representation is irreducible if and only if the only closed subspaces $K$ of $H$ which are {\it invariant for the representation}, i.e., $S_iK\subseteq K$, $S_i^*K\subseteq K$, for all $i\in\{0,\dots,N-1\}$, are $K=\{0\}$ and $K=H$.

For a closed subspace $K$ of $H$ we will denote by $P_K$ the orthogonal projection onto $K$.

\end{definition}

\begin{definition}\label{p2}
Let $H$ be a Hilbert space. A family of vectors $\{e_i: i\in I\}$ in $H$ is called a {\it frame}, if there exists constants $A,B>0$, called {\it the frame bounds} such that 
$$A\|v\|^2\leq \sum_{i\in I}|\ip{v}{e_i}|^2\leq B\|v\|^2,\mbox{ for all }v\in H.$$
The family of vectors is called a {\it Parseval frame} if the frame bounds are equal to 1, $A=B=1$. 

\end{definition}

\begin{theorem}\cite[Theorem 5.1]{BJK00}\label{dil}
Let $K$ be a Hilbert space, and let $V_0,\dots ,V_{N-1}$ be bounded operators satisfying
\begin{equation}\label{e0}
\displaystyle\sum_{i=0}^{N-1}V_iV^*_i=I_K
\end{equation}
Then $K$ can be embedded into a larger Hilbert space $H$ carrying a representation $S_0,\dots ,S_{N-1}$ of the Cuntz algebra $\mathcal{O}_N$ such that $K$ is cyclic for the representation, and if $P_K:H\to K$ is the projection onto $K$ we have
\begin{equation}\label{e1}
S^*_i(K)\subset K, \text{ and }V^*_iP_K=S^*_iP_K=P_KS^*_iP_K,\mbox{ so } P_KS_iP_K=P_KS_i,\quad V_i=P_KS_{i_{_{|K}}}.
\end{equation}

The system $(H, S_i, P_K)$ is unique up to unitary equivalence, and if $\sigma:\mathcal{B}(K)\to\mathcal{B}(K)$ is defined by
$$\sigma(A):=\sum_{i=0}^{N-1} V_iAV^*_i,$$
then the commutant $\{S_i, i=0,\dots ,N-1\}'$ is isometrically order isomorphic to the fixed point set $\mathcal{B}(K)^{\sigma}=\{A\in\mathcal{B}(K)\text{  }|\text{  }\sigma(A)=A\}$, by the map $A\mapsto P_KAP_K$.
\end{theorem}

%\begin{remark} (Gabriel): In \cite{BJK00} the proof of the theorem above uses Stinespring dilation of the completely positive map $\phi:\mathcal{O}_N\to \mathcal{B}(K)$, $\phi(s_Is^*_J)=V_IV^*_J$, where $s_I$ represents finite length compositions of the fixed set of isometries $\{s_i, i=0,\dots ,N-1\}$ generating $\mathcal{O}_N$. Then $K$ is identified with its embedding in $H$. It seemed to me the details for (\ref{e1}) are not fully supplied ( their Theorem 2.1 is also mentioned in the proof). What is not clear to me is why they get $S^*_i(K)\subset K$, else all is good (both $pS_i$ and $S^*_ip$ leave $K$ invariant therefore $pS_i$ commmutes with $p$).
%\end{remark}

\begin{definition}\label{def1.1}
Let $K$ be a Hilbert space, $N\geq2$ an integer, and $V_i$ some bounded operators on $K$, $i=0,1,\dots,N-1$. We say that $(K,V_i)_{i=0}^{N-1}$ is a {\it row co-isometry} if
$$\sum_{i=0}^{N-1}V_iV_i^*=I_K.$$
\end{definition}
\begin{definition}
The data in Theorem \ref{dil} will be referred to as the Cuntz dilation $(H, S_i)_{i=0}^{N-1}$ corresponding to the row co-isometry $(K, V_i)_{i=0}^{N-1}$.
\end{definition}

\begin{definition}\label{def2.6}
Let $\ty$ denote the empty word. Let $\Omega$ be the set of all finite words with digits in $\{0,\dots,N-1\}$, including the empty word. For $\omega\in\Omega$, we denote by $|\omega|$ the length of $\omega$. 
 
For a word $\omega:=\omega_1\dots\omega_n\in\Omega$, we will denote by $V_\omega:=V_{\omega_1}\dots V_{\omega_n}$ and similarly for $S_\omega$. Also $V_\ty=I$ and $S_\ty=I$.
\end{definition}

\begin{definition}\label{defm1}
Let $(K,V_i)_{i=0}^{N-1}$ be a row co-isometry. Let $K_0$ be a closed subspace of $K$. Suppose the following conditions are satisfied:
\begin{enumerate}
	\item [(i)] $V_i^*K_0\subset K_0$ for all $i\in\{0,\dots,N-1\}$.
	\item[(ii)] There exists an orthonormal basis $(e_c)_{c\in M}$ for $K_0$, with $M$ finite, some maps $\nu_i:M\rightarrow \bc$, $g_i:M\rightarrow M$, $i\in\{0,\dots,N-1\}$ such that $V_i^*e_c=\nu_i(c)e_{g_i(c)}$, for all $i\in\{0,\dots,N-1\}$ and $c\in M$. 
	\item[(iii)] The maps $g_i$ are one-to-one whenever the transitions are possible, in the sense that, for all $i\in\{0,\dots,N-1\}$ and all $c_1,c_2\in M$ such that $\nu_i(c_1)\neq 0$ and $\nu_i(c_2)\neq 0$, if $g_i(c_1)=g_i(c_2)$ then $c_1=c_2$. 
\end{enumerate}
Then we say that $(V_i)_{i=0}^{N-1}$ {\it acts on $K_0$ as a random walk on the graph $(M,\{g_i\},\{\nu_i\})$}. $M$ is the set of vertices, the maps $g_i$ indicate an edge with label $i$, from a vertex $c\in M$ to the vertex $g_i(c)$, and $\nu_i(c)$ represents a weight for this edge, more precisely $|\nu_i(c)|^2$ is a probability of transition from $c$ to $g_i(c)$ along this edge. 

 Indeed, since $\sum_{i=0}^{N-1}\ip{V_i^*e_c}{V_i^*e_c}=\|e_c\|^2,$ it follows that $\sum_{i=0}^{N-1}|\nu_i(c)|^2=1$.

Note that if there are exactly two distinct $i,i'\in \{0,\ldots, N-1\} $ such that $g_i(c) = g_{i'}(c)$, then the total probability of transition from $c$ to $g_i(c)$ is $|\nu_i(c)|^2 + |\nu_{i'}(c)|^2$. We will use the convention that when we say ``the probability of transition from $c$ to $g_i(c)$,'' we mean the probability of transition from $c$ to $g_i(c)$ (only) through $i$, that is, $|\nu_i(c)|^2$.

\begin{itemize}
\item[$\bullet$] If, for any $c_1,c_2\in M$, there exists $\omega=\omega_1\dots\omega_n$ in $\Omega$ such that $g_{\omega}(c_1):=g_{\omega_n}\circ g_{\omega_{n-1}}\circ \dots \circ g_{\omega_1}(c_1)=c_2$ and $\nu_\omega(c_1):= \nu_{\omega_1}(c_1)\nu_{\omega_2}(g_{\omega_1}(c_1))\dots\nu_{\omega_n}(g_{\omega_{n-1}}\dots g_{\omega_1}(c_1))\neq 0$, then we say that the random walk is {\it irreducible}. In other words, one can reach $c_2$ from $c_1$ through $\omega$ with positive probability $|\nu_\omega(c_1)|^2$.

\item[$\bullet$] We say that $(V_i)_{i=0}^{N-1}$ is {\it simple} on $K_0$ if the only operators $T:K_0\rightarrow K_0$ with 
\begin{equation}
P_{K_0}\left(\sum_{i=0}^{N-1}V_iTV_i^*\right)P_{K_0}=T
\label{eqdefm1_1}
\end{equation}
are $T=\lambda I_{K_0}$, $\lambda\in\bc$; recall that $P_{K_0}$ denotes the orthogonal projection onto the subspace $K_0$.

\item[$\bullet$] If, for any $c\in M$ and $i\in\{0,\dots,N-1\}$, we have $V_i(e_{g_i(c)})=\cj\nu_i(c)e_c$, whenever $\nu_i(c)\neq 0$, we say that $\{V_i\}$ is {\it reversing } on the random walk. 

\item[$\bullet$] If, for any $c_1\neq c_2$ in $M$, there exists $n$ such that, for all $\omega\in \Omega$ with $|\omega|\geq n$, we have $\nu_{\omega}(c_1)=0$ or $\nu_{\omega}(c_2)=0$, then we say that the random walk is {\it separating}. 

\item[$\bullet$] A word $\beta=\beta_0\dots\beta_{p-1}\in\Omega$, $\beta\neq \ty$, is called a {\it cycle word } for $c\in M$, if $g_\beta(c)=c$, and $g_{\beta_k}g_{\beta_{k-1}}\dots g_{\beta_0}(c)\neq c$ for $0\leq k<p-1$ and $\nu_{\beta}(c)\neq 0$. The points $g_{\beta_k}g_{\beta_{k-1}}\dots g_{\beta_0}(c)$, $k=0,\dots,p-1$ are called {\it cycle points}.

\item[$\bullet$] A word $\beta\in\Omega$ is called a {\it loop} for $c\in M$ if $\beta=\ty$, or $g_{\beta}(c)=c$ and $\nu_{\beta}(c)\neq 0$. Note that any loop $\beta\neq \ty$ is of the form $\beta=\beta_1\dots\beta_p$, for some cycle words $\beta_1,\dots,\beta_p$ for $c$.
\end{itemize}
\end{definition}

In Section \ref{secex}, we give several examples of such co-isometries, which include Fourier series on fractal measures, Walsh bases, and a combination of the two. 
%\begin{definition}\label{defm2}

\begin{definition}\label{defp3}
Let $X=(X,d)$ be a complete metric space. A map $g:X\rightarrow X$ is called a {\it contraction}, if there exists a constant $0<c<1$, such that 
$$d(g(x),g(y))\leq c d(x,y),\mbox{ for all }x,y\in X.$$

An {\it iterated function system (IFS)} is a finite family of contractions $\{g_i\}_{i=0}^{N-1}$ on $X$. By \cite[Section 3]{Hut81}, given an iterated function system $\{g_i\}_{i=0}^{N-1}$ on $X$, there exists a unique closed bounded set $X_0$ such that 
$$X_0=\bigcup_{i=0}^{N-1}g_i(X_0).$$
Furthermore $X_0$ is compact. $X_0$ is called {\it the attractor} of the iterated function system. 

\end{definition}

\section{Main Results}

%If, for any $c_1,c_2\in M$, there exists $\omega=\omega_1\dots\omega_n$ in $\Omega$ such that $g_{\omega}(c_1):=g_{\omega_n}g_{\omega_{n-1}}\dots g_{\omega_1}(c_1)=c_2$ and $\nu_\omega(c_1)\neq 0$, then we say that the random walk is {\it irreducible}.
%
%
%If, for any $c_1\neq c_2$ in $M$, there exists $n$ such that, for all $\omega\in \Omega$ with $|\omega|\geq n$, we have $\nu_{\omega}(c_1)=0$ or $\nu_{\omega}(c_2)=0$, then we say that the random walk is {\it separating}. 
%\end{definition}

Structured bases regularly arise as the orbit under the action of a system of isometries applied to a collection of vectors.  Since frames are compressions of orthonormal bases, structured frames might analogously arise under the action of a system of co-isometries.  In the case of row co-isometries, the question of whether their action generates a frame depends on the action of the co-isometries themselves as well as the structure of their Cuntz dilation.  We describe sufficient conditions under which frames are generated by row co-isometries in our main results, Theorems \ref{thm1} and \ref{th2.11}.

\begin{theorem}\label{thm1}
Let $(K,V_i)_{i=0}^{N-1}$ be a row co-isometry and let $(H,S_i)_{i=0}^{N-1}$ be its Cuntz dilation. Suppose there is a subspace $K_0$ of $K$ such that $\{V_i\}$ acts on $K_0$ as an irreducible random walk $(M,\{g_i\},\{\nu_i\}_i)$ with orthonormal basis $(e_c)_{c\in M}$. Assume in addition that $\{V_i\}$ is simple on $K_0$.

Fix a point $c\in M$. Define 
$$\Omega_c^{(0)}:=\left\{\omega\in\Omega :\omega\mbox{ does not end in a cycle word for $c$}\right\}.$$
(In particular, $\ty \in\Omega_c^{(0)}$.)

For $n\geq 1$, define $\Omega_c^{(n)}$ to be the set of words that end in exactly $n$ cycle words for $c$, i.e., words $\omega\in\Omega$ of the form $\omega=\omega_0\beta_1\dots\beta_n$, $\omega_0\in\Omega_c^{(0)}$, $\beta_1,\dots\beta_n$ cycle words for $c$.

Let $$\mathcal E_c^{(n)}:=\left\{S_\omega e_c : \omega\in \Omega_c^{(n)}\right\},\quad H_c^{(n)}:=\Span\, \mathcal E_c^{(n)},\quad (n\geq 0).$$
Then

\begin{enumerate}
	\item[(a)] $\mathcal E_c^{(n)}$ is an orthonormal basis for $H_c^{(n)}$, for all $n\geq 0$.
	\item[(b)] $H_c^{(n)}\subseteq H_c^{(n+1)}$ for all $n\geq 0$.
	\item[(c)] $\cup_{n\geq0}\mathcal E_c^{(n)}=\{S_\omega e_c : \omega\in \Omega\}$.
	\item[(d)] Let $V:=\Span\{S_\omega e_c :\omega\in\Omega\}=\overline{\cup_{n\geq 0}H_c^{(n)}}$. Then $V$ is invariant for the Cuntz representation $\{S_i\}$. Also $(V, S_i|_V)_{i=0}^{N-1}$ is the Cuntz dilation of the row co-isometry $(K_0,P_{K_0}V_iP_{K_0})_{i=0}^{N-1}$ and it is irreducible.

\end{enumerate}

In addition, the following statements are equivalent

\begin{enumerate}
	\item[(i)] $\Span\{S_\omega e_c :\omega\in \Omega\}=H$.
	\item[(ii)] The Cuntz representation $(H,S_i)_{i=0}^{N-1}$ is irreducible.
	\item[(iii)] The only operators $T:K\rightarrow K$ with 
	$$\sum_{i=0}^{N-1}V_iTV_i^*=T,$$
	are $T=\lambda I_K$, $\lambda\in\bc$.
	\end{enumerate}
	
    If $\{V_i\}$ is reversing, then we also have:
	\begin{enumerate}
		\item[(e)] For all $n\geq 0$ and all $v\in K$,
		\begin{equation}
\left\|P_{H_c^{(n)}}v \right\|^2=\sum_{\omega\in\Omega_c^{(0)}}\left|\ip{V_{\omega_0}e_c}{v}\right|^2=\left\|P_{H_c^{(0)}}v \right\|^2=\left\|P_V v\right\|^2.
\label{eqt4}
\end{equation}

	\end{enumerate}
and the previous statements are equivalent to 
	
	\begin{enumerate}
		\item [(iv)] $\{V_\omega e_c : \omega\in\Omega_c^{(0)}\}$ is a Parseval frame for $K$. 
	\end{enumerate}
\end{theorem}

For the second part of the paper, we will impose some extra conditions on the co-isometry $(K,V_i)_{i=0}^{N-1}$.

\begin{assumptions}\label{as1}

Let $(K,V_i)_{i=0}^{N-1}$ be a row co-isometry. We make the following assumptions

\begin{equation}
\mbox{ There exists a continuous map $e:\mathcal T\rightarrow K$ from a complete metric space $\mathcal T$, such that}
\label{eqas1}
\end{equation}
$\Span\{ e(t) : t\in\mathcal T\}=K$, and $\|e_t\|=1$ for all $t\in\mathcal T$.
\begin{equation}
\mbox{For any $i\in\{0,\dots,N-1\}$ and $t\in\mathcal T$, } V_i^*e_t=\nu_i(t)e_{g_i(t)},\mbox{ for some $\nu_i(t)\in\bc$ and $g_i(t)\in\mathcal T$.}
\label{eqas2}
\end{equation}
where $e_t:=e(t)$.
\begin{equation}
\mbox{ The maps $g_i$ are contractions.}
\label{eqas3}
\end{equation}

\end{assumptions}

Under these assumptions, one can define a random walk on the set $\mathcal T$, where the transition from $t$ to $g_i(t)$ is given probability $|\nu_i(t)|^2$. Note that, 
$$1=\|e_t\|^2=\sum_{i=0}^{N-1}\ip{V_i^*e_t}{V_i^*e_t}=\sum_{i=0}^{N-1}|\nu_i(t)|^2\ip{e_t}{e_t}=\sum_{i=0}^{N-1}|\nu_i(t)|^2.$$
An important role in the study of the Cuntz dilation associated to the co-isometry $\{V_i\}$ and the Parseval frames generated by it, is played by the minimal invariant sets associated to this random walk. We define these and more here.

\begin{definition}\label{def3.4}
For a point $t\in\mathcal T$, and $i\in\{0,\dots,N-1\}$, we say that the transition $t\rightarrow g_i(t)$ is {\it possible through} $i$ if $\nu_i(t)\neq 0$. We also write $t\stackrel{i}{\rightarrow} g_i(t)$ or $t\rightarrow g_i(t)$.

\begin{remark}\label{rem2.9}
    Note that if $t' = g_i(t) = g_{i'}(t)$, the transition from $t\rightarrow t'$ may be possible through either $i$ or $i'$. We will make the convention that if we write ``$t\rightarrow g_i(t)$ is possible'', we mean ``$t\rightarrow g_i(t)$ is possible through $i$,'' to distinguish the paths along $i$ and $i'$.
\end{remark}

Given a word $\omega=\omega_1\dots\omega_n$, we define $g_\omega=g_{\omega_n}\circ\dots\circ g_{\omega_1}$. We say that the transition $t\rightarrow g_\omega(t)$ is {\it possible (in several steps) through $\omega$}, if all the transitions $t\rightarrow g_{\omega_1}(t)\rightarrow g_{\omega_2}g_{\omega_1}(t)\rightarrow\dots\rightarrow g_{\omega_n}\dots g_{\omega_1}(t)$ are possible.

We define
$$\nu_{\omega}(t)=\nu_{\omega_1}(t)\nu_{\omega_2}(g_{\omega_1}(t))\dots\nu_{\omega_n}(g_{\omega_{n-1}}\dots g_{\omega_1}(t)).$$
$|\nu_\omega(t)|^2$ is the probability of transition from $t$ to $g_\omega(t)$ through $\omega$ (passing through\\ $g_{\omega_1}(t)$, $g_{\omega_2}g_{\omega_1}(t),\dots, g_{\omega_n}\dots g_{\omega_1}(t)$). 

Note that the transition $t\rightarrow g_\omega(t)$ is possible in several steps, if and only if $\nu_\omega(t)\neq 0$.

A subset $M$ of $\mathcal T$ is called {\it invariant}, if, for all $t$ in $M$ and $i\in\{0,\dots,N-1\}$, if the transition $t\rightarrow g_i(t)$ is possible, then $g_i(t)$ is in $M$. We define the orbit of $t\in \mathcal T$ to be 
$$\mathcal O(t)=\left\{ t'\in\mathcal T : \mbox{ There exists $\omega\in\Omega$ such that the transition $t\rightarrow g_\omega(t)=t'$ is possible}\right\}.$$

A closed invariant subset $M$ of $\mathcal{T}$ is {\it minimal} if there does not exist any proper closed invariant subset of $M$. Equivalently, as in Lemma \ref{lem4.11-0}, $M=\cj{\mathcal O(x_0)}$ for any $x_0\in M$.

Given an invariant set $M$, we define the subspaces $K(M)$ of $K$ and $H(M)$ of $H$ by
$$K(M)=\Span\{ e_t : t\in M\},$$
$$H(M)=\Span\{S_\omega e_t : t\in M, \omega\in\Omega\}.$$
\end{definition}

\begin{theorem}\label{th3.5}
Suppose that the assumptions \eqref{eqas1}--\eqref{eqas3} hold. Then there are finitely many minimal compact invariant sets.  If $M_1,\dots, M_p$ is the complete list of minimal compact invariant sets, then the spaces $K(M_j)$ are invariant for $V_i^*$, $i=0,\dots, N-1$ and the spaces $H(M_j)$ are invariant for the representation $(S_i)_{i=0}^{N-1}$.
The spaces $H(M_j)$ are mutually orthogonal and 
$$\bigoplus_{j=1}^p H(M_j)=H.$$
The sub-representations $(S_i)_{i=0}^{N-1}$ on $H(M_j)$, $j=1,\dots,p$ are irreducible and disjoint. (Recall that two representations $(H,S_i)_i$, $(H',S_i')_i$ are called disjoint if there is no non-zero intertwining operator $W:H\rightarrow H'$, $WS_i=S_i'W$, $WS_i^*=S_i'^*W$ for all $i$.)

\end{theorem}

\begin{theorem}\label{th2.11}
Assume in addition that the maps $\{g_i\}$ are one-to-one and that all the minimal compact invariant sets are finite.  Pick a point $c_j$ in $M_j$ for every $j=1,\dots,p$. Then $\{V_i\}$ acts as an irreducible, separating random walk $(M_j,\{g_i|_{M_j}\},\{\nu_i|_{M_j}\})$ on the spaces $K(M_j)$. If $\{V_i\}$ is reversing on each $K(M_j)$, then $\{V_\omega e_{c_j} : \omega\in\Omega_{c_j}^{(0)},j=1,\dots,p-1\}$ is a Parseval frame for $K$.
\end{theorem}

\section{Proofs}

\begin{proof}[Proof of Theorem \ref{thm1}]

We begin with a general lemma which shows that projections of iterations of the Cuntz dilation isometries are the corresponding iterations of the row co-isometries. 

\begin{lemma}\label{l1} Let $(H, S_i)_{i=0}^{N-1}$ be the Cuntz dilation of the system $(K, V_i)_{i=0}^{N-1}$, and $P_K:H\to K$ the projection onto $K$. 
For $\omega\in\Omega$,  
\begin{equation}\label{e3}
V_{\omega}(k)=P_KS_{\omega}(k)\text{ for all }k\in K
\end{equation}
\end{lemma}

\begin{proof}
Let $\omega=\omega_1\omega_2$ and $k\in K$. Using \eqref{e1} and $P_KS_i(k)\in K$, we have
$$V_{\omega_1}V_{\omega_2}(k)=V_{\omega_1} P_KS_{\omega_2}(k)
=P_KS_{\omega_1} P_KS_{\omega_2}(k)=P_KS_{\omega_1}S_{\omega_2}(k)
$$
By induction, \eqref{e3} follows.
\end{proof}

We will also need the next general result which shows when the vectors $S_\omega v$ are orthogonal. 

\begin{lemma}\label{lemo}
If $(H,S_i)_{i=0}^{N-1}$ is a representation of the Cuntz algebra, $v,v'\in H$, $\omega,\omega'\in\Omega$ and $\ip{S_{\omega}v}{S_{\omega'}v'}\neq 0$, then $\omega$ is a prefix of $\omega'$ or vice-versa, i.e., there exists a word $\beta\in\Omega$ such that $\omega'=\omega\beta$ or $\omega=\omega'\beta$. 
\end{lemma}

\begin{proof}
Let $n=|\omega|$, $n'=|\omega'|$. If there exists $k$ such that $\omega_k\neq \omega'_k$ then take the smallest such $k$, and we have 
$$\ip{S_{\omega_1\dots\omega_{k-1}\omega_k\dots\omega_n}v}{S_{\omega_1'\dots\omega_{k-1}'\omega_k'\dots\omega_{n'}'}v'}
=\ip{S_{\omega_k\dots\omega_n}v}{S_{\omega_k'\dots\omega_{n'}'}v'}=0.$$
This is impossible, therefore $\omega$ is a prefix of $\omega'$ or vice-versa. 
\end{proof}
To prove item (a) in the theorem, we will use the next Lemma:

\begin{lemma}\label{lemm2}
If $\omega_1,\omega_2\in\Omega$ and $\ip{S_{\omega_1}e_c}{S_{\omega_2} e_c}\neq 0$, then there exists a loop $\beta$ for $c$ such that $\omega_1=\omega_2\beta$ or $\omega_2=\omega_1\beta$. 

\end{lemma}

\begin{proof}

First, an easy computation shows that, for $c\in M$ and $\omega\in\Omega$, 
\begin{equation}
V_\omega^*e_c=S_\omega^*e_c=\nu_\omega(c)e_{g_\omega(c)}.
\label{eqlem2.1}
\end{equation}
By Lemma \ref{lemo}, either $\omega_1=\omega_2\beta$ or $\omega_2=\omega_1\beta$ for some $\beta\in\Omega$. Assume the first. We prove that $\beta$ is a loop. We have
$$0\neq \ip{S_{\omega_1}e_c}{S_{\omega_2} e_c}=\ip{S_{\omega_2\beta}e_c}{S_{\omega_2}e_c}=\ip{S_\beta e_c}{e_c}=\ip{e_c}{S_\beta^* e_c}=\cj\nu_\beta(c)\ip{e_c}{e_{g_\beta(c)}}.$$
This implies that $g_\beta(c)=c$ and $\nu_\beta(c)\neq 0$ so $\beta$ is a loop for $c$.

\end{proof}

\begin{lemma}\label{leminj}
Suppose $\omega_1\beta_1=\omega_2\beta_2$ for some cycle words $\beta_1$ and $\beta_2$ for $c$ and some words $\omega_1$ and $\omega_2$; then $\beta_1=\beta_2$ and $\omega_1=\omega_2$. Also, every word can be written uniquely as $\omega_0\beta_1\beta_2\dots\beta_n$ for some cycle words $\beta_1,\dots,\beta_n$ for $c$, ($n$ could be 0) and some word $\omega_0\in \Omega_c^{(0)}$.

\end{lemma}

\begin{proof}
Suppose $|\beta_1|>|\beta_2|$. Then, if we read the words $\omega_1\beta_1=\omega_2\beta_2$ from the right, we see that $\beta_1=\gamma\beta_2$ for some non-empty word $\gamma$. We have $g_{\beta_2}(c)=c=g_{\beta_1}(c)=g_{\beta_2}(g_\gamma(c))$. Since all these transitions are possible and the maps $g_i$ are one-to-one in this case, it follows that $g_\gamma(c)=c$, but this contradicts the fact that $\beta_1$ is a cycle word for $c$.  Thus $\beta_1=\beta_2$ and $\omega_1=\omega_2$.

Now take an arbitrary word $\omega$. If it does not end in a cycle word, it is in $\Omega_c^{(0)}$. If it ends in a cycle word, by the previous statement, it can be written uniquely as $\omega=\omega_1\beta_1$, with $\beta_1$ cycle word. Repeat for $\omega_1$ and use induction, to obtain the last statement in the lemma. 

See Remark \ref{remmc}, where we point out why the condition that the maps $g_i$ are injective is important in this context. 
\end{proof}

Now take $\omega_1\neq \omega_2$ in $\Omega_c^{(n)}$ for some $n\geq 0$. If $\ip{S_{\omega_1}e_c}{S_{\omega_2}e_c}\neq 0$ then, with Lemma \ref{lemm2}, there exists a loop $\beta\neq\ty$ such that $\omega_2=\omega_1\beta$ or $\omega_1=\omega_2\beta$. Assume that $\omega_1=\omega_2\beta$. Since $\beta$ is a loop, $\beta\neq \ty$, we have that $\beta=\beta_1\dots\beta_p$ for some cycle words $\beta_1,\dots,\beta_p$ for $c$. Since $\omega_2$ ends in $n$ cycle words for $c$, it follows that $\omega_1=\omega_2\beta$ ends in $n+p$ cycle words for $c$, which, by Lemma \ref{leminj}, is impossible since $\omega_1\in\Omega_c^{(n)}$. This shows that $\mathcal E_c^{(n)}$ is an orthonormal basis for $H_c^{(n)}$.

\begin{remark}\label{remmc}

The condition that the maps $g_i$ be one-to-one when the transitions are possible is important in Lemma \ref{leminj}, and to guarantee that the vectors $S_\omega e_c$ with $e_c$ in $\Omega_c^{(n)}$ are orthogonal. Indeed, if we do not assume this condition, it is possible that a word ends both in a cycle word for $c$ and in two cycle words for $c$. Consider 
$$V_0^*:=\begin{bmatrix}
\frac{1}{\sqrt{2}}&\frac{1}{\sqrt{2}}\\
0&0
\end{bmatrix},\quad
V_1^*:=\begin{bmatrix}
0&0\\
\frac{1}{\sqrt{2}}&-\frac{1}{\sqrt{2}}
\end{bmatrix}.$$

A simple check shows that $(V_0,V_1)$ is a row co-isometry on $\bc^2$, and if $\{e_1,e_2\}$ is the standard basis for $\bc^2$, then, this co-isometry acts as a random walk on $M=\{1,2\}$, with $g_0(1)=1$, $g_0(2)=1$, $g_1(1)=2$, $g_1(2)=2$, $\nu_0(1)=\frac{1}{\sqrt2}$,  $\nu_0(2)=\frac{1}{\sqrt2}$, $\nu_1(1)=\frac{1}{\sqrt2}$, $\nu_1(2)=-\frac{1}{\sqrt2}$. 

The random walk is clearly irreducible and, if $T=\begin{bmatrix} T_{11}& T_{12}\\ T_{21}&T_{22}\end{bmatrix}$ with $T=V_0TV_0^*+V_1TV_1^*$, we obtain that $T=\frac{1}{2}(T_{11}+T_{22})I_2$, so the co-isometry is simple on $\bc^2$. 

Note that both $0$ and $10$ are cycle words for $c:=1\in M$. Therefore the word $\omega=1010$ ends the cycle word $0$ for $c$, since $1$, $01$ and $101$ are not cycle words for $c$. At the same time $\omega$ ends in two cycle words for $c$, $\omega=(10)(10)$. 
\end{remark}

(b) We will need some key lemmas.
\begin{lemma}\label{lem4.14}
Define the random walk/Markov chain $X_n$, $n\geq 0$ on $M$ by 
$$\Prob(X_n=g_i(t) | X_{n-1}=t)=|\nu_i(t)|^2,\quad (t\in M).$$
Then 
$$\Prob(X_n=g_{\omega_n}\dots g_{\omega_1}(t),X_{n-1}=g_{\omega_{n-1}}\dots g_{\omega_1}(t),\dots, X_1=g_{\omega_1}(t)| X_0=t)=|\nu_{\omega_1\dots\omega_n}(t)|^2.$$
Also the random walk is recurrent. 
\end{lemma}

\begin{proof}
The first equality follows from a simple computation. 

Since $M$ is finite and the random walk is irreducible by hypothesis (see Definition \ref{defm1}), the random walk is also recurrent (see, e.g., \cite[Theorem 6.4.4]{Dur10} or \cite[Section XV.6]{Fel}).
\end{proof}

\begin{definition}\label{def4.15}
Let $t,c\in M$ and $\omega\in\Omega$, $\omega=\omega_1\dots\omega_n\neq\ty$. We say that $g_\omega(t)=c$ {\it for the first time} if $g_\omega(t)=c$, $\nu_\omega(t)\neq 0$, and $g_{\omega_1\dots\omega_k}(t)\neq c$ for $1\leq k<n$. In particular, if $t=c$ and $\omega\neq\ty$, then $g_\omega(t)=c$ for the first time if and only if $\omega$ is a cycle word for $c$.
\end{definition}

\begin{lemma}\label{lem4.15}
Let $t,c\in M$. Then 
\begin{equation}
\sum_{g_\omega(t)=c\mbox{ for the first time}}|\nu_\omega(t)|^2=1.
\label{eq4.15.1}
\end{equation}
\end{lemma}

\begin{proof}
By Lemma \ref{lem4.14}, the random walk on $M$ is recurrent, so $\Prob(X\mbox{ ever enters }c |X_0=t)=1$. But
$$\Prob(X\mbox{ ever enters }c |X_0=t)=\sum_{n=1}^\infty\Prob(X_n=c,X_{n-1}\neq c,\dots, X_1\neq c | X_0=t)$$$$=
\sum_{n=1}^\infty \sum_{\omega = \omega_1\ldots\omega_n}\Prob(X_n=g_{\omega_n}\dots g_{\omega_1}(t)=c, X_k=g_{\omega_k}\dots g_{\omega_1}(t)\neq c\mbox{ for }1\leq k<n |X_0=t)$$
$$=\sum_{g_\omega(t)=c\mbox{ for the first time}}|\nu_\omega(t)|^2.$$

\end{proof}

\begin{lemma}\label{lem4.16}
Let $t,c\in M$. Then 
$$e_t=\sum_{g_\omega(t)=c\mbox{ for the first time }}\nu_\omega(t)S_\omega e_c.$$
\end{lemma}

\begin{proof}

Using the Cuntz relations, and that $S_i^*e_x=\nu_i(x)e_{g_i(x)}$ for all $i=0,\dots, N-1$ and $x\in \mathcal T$, we have 
$$e_t=\sum_{\omega_1=0}^{N-1}S_{\omega_1}S_{\omega_1}^*e_t=\sum_{\omega_1}\nu_{\omega_1}(t)S_{\omega_1}e_{g_{\omega_1}(t)}=\sum_{\omega_1: g_{\omega_1}(t)=c}\nu_{\omega_1}(t)S_{\omega_1}e_c+\sum_{\omega_1 : g_{\omega_1}(t)\neq c}\nu_{\omega_1}(t)S_{\omega_1}e_{g_{\omega_1}(t)}$$
$$=\sum_{\omega_1: g_{\omega_1}(t)=c}\nu_{\omega_1}(t)S_{\omega_1}e_c+\sum_{\omega_1 : g_{\omega_1}(t)\neq c}\nu_{\omega_1}(t)\sum_{\omega_2}S_{\omega_1}S_{\omega_2}S_{\omega_2}^*e_{g_{\omega_1}(t)}$$
$$=\sum_{\omega_1: g_{\omega_1}(t)=c}\nu_{\omega_1}(t)S_{\omega_1}e_c+\sum_{\omega_1\omega_2 : g_{\omega_1}(t)\neq c,g_{\omega_1\omega_2}(t)=c}\nu_{\omega_1\omega_2}(t)S_{\omega_1\omega_2}e_{c}+$$$$
\sum_{\omega_1\omega_2: g_{\omega_1}(t)\neq c,g_{\omega_1\omega_2}(t)\neq c}\nu_{\omega_1\omega_2}(t)S_{\omega_1\omega_2}e_{g_{\omega_1\omega_2}(t)}.$$
Continuing, by induction we get, for $n\in\bn$,
\begin{equation}
e_t=\sum_{\mtiny{1\leq |\omega|\leq n, g_\omega(t)=c \mbox{ for the first time}}}\nu_\omega(t)S_\omega e_c+\sum_{\mtiny{\substack{|\omega|=n\\ g_{\omega_1\dots\omega_k}(t)\neq c\\ \mbox{ for any }1\leq k\leq n}}}\nu_\omega(t)S_\omega e_{g_\omega(t)}
\label{eq4.16.1}
\end{equation}
Note that if we take two different $\omega_1,\omega_2$ that appear in the first sum, then one can not be the prefix of the other, by the definition of ``$g_\omega(t)=c$ for the first time''.  Also if $\omega_1$ appears in the first sum, and $\omega_2$ appears in the second, then $\omega_1$ is not a prefix of $\omega_2$ because $g_{\omega_1}(t)=c$. Using Lemma \ref{lemo}, it follows that all the terms in the first sum are mutually orthogonal, and they are orthogonal to all the terms in the second sum. Denote the second sum by $A_n$. Since the norms are equal, we get 
\begin{equation}
1=\sum_{\mtiny{1\leq |\omega|\leq n, g_\omega(t)=c \mbox{ for the first time}}}|\nu_\omega(t)|^2+\|A_n\|^2.
\label{eq4.16.2}
\end{equation}
From Lemma \ref{lem4.15}, 
$$\lim_{n\rightarrow\infty}\sum_{\mtiny{1\leq |\omega|\leq n, g_\omega(t)=c \mbox{ for the first time}}}|\nu_\omega(t)|^2=1,$$
therefore $A_n\rightarrow 0$ and this proves the lemma.

\end{proof}

Let $n\geq 0$ and $\omega\in\Omega_c^{(n)}$, then with Lemma \ref{lem4.16}

$$S_\omega e_c=S_\omega\left(\sum_{\beta\mbox{ cycle word for $c$}}\nu_\beta(c)S_\beta e_c\right)=\sum_\beta \nu_{\beta}(c)S_{\omega\beta}e_c.$$
But $\omega\beta\in\Omega_c^{(n+1)}$ so $S_\omega e_c \in H_c^{(n+1)}$. Thus $H_c^{(n)}\subseteq H_c^{(n+1)}$. 

(c) is trivial.

(d) Using Lemma \ref{lem4.16}, it follows that, for every $t\in M$, $e_t\in V$, thus $K_0\subseteq V$. 
Clearly $V$ is invariant for all the operators $S_i$. To see that it is also invariant for the operators $S_i^*$, let $\omega\in\Omega$. If $\omega=\omega_1\dots\omega_n\neq\ty$, then $S_i^*S_\omega e_c=\delta_{i\omega_1} S_{\omega_2\dots\omega_n} e_c\in V$. If $\omega=\ty$, then $S_i^*e_c=\nu_i(c)e_{g(c)}\in V$, according to the previous argument. Thus $V$ is invariant for the Cuntz representation and it contains $K_0$.

Since $K_0$ is invariant for all $V_i^*$, we have $P_{K_0}V_i^*P_{K_0}=V_i^*P_{K_0}$ and $P_{K_0}V_iP_{K_0}=P_{K_0}V_i$. Let $W_i:=P_{K_0}V_iP_{K_0}$. Then 
$$\sum_{i=0}^{N-1}W_iW_i^*=P_{K_0}.$$

Since $\{S_\omega e_c :\omega \in\Omega\}$ spans $V$ it follows that $K_0$ is cyclic for the representation $\{S_i\}$. Thus $(V, S_i|_V)_{i=0}^{N-1}$ is the Cuntz dilation of the row co-isometry $(K_0,W_i)_{i=0}^{N-1}$. By Theorem \ref{dil}, it follows that the Cuntz representation $(V,S_i|_V)_{i=0}^{N-1}$ is irreducible, since $\{V_i\}$ is simple on $K_0$.

(e) We assume now that $\{V_i\}$ is reversing. If $\omega\in\Omega_c^{(n)}$, then $\omega=\omega_0\beta_1\dots\beta_n$, with $\omega_0\in\Omega_c^{(0)}$ and $\beta_1,\dots,\beta_n$ cycle words for $c$. 

If $\beta$ is a cycle word for $c$, then $V_\beta e_c=V_\beta e_{g_\beta(c)}=\cj\nu_\beta(c) e_c$. So 
$$V_\omega e_c=V_{\omega_0} V_{\beta_1}\dots V_{\beta_n}e_c=\cj\nu_{\beta_n}(c)\cj\nu_{\beta_{n-1}}(c)\dots\cj\nu_{\beta_1}(c)V_{\omega_0}e_c.$$

Let $v\in K$. Then, with (a), we have 

$$\left\|P_{H_c^{(n)}}v \right\|^2=\sum_{\omega\in\Omega_c^{(n)}}\left|\ip{S_\omega e_c}{v}\right|^2=\sum_{\omega\in\Omega_c^{(n)}}\left|\ip{P_KS_\omega e_c}{v}\right|^2
=\sum_{\omega\in\Omega_c^{(n)}}\left|\ip{V_\omega e_c}{v}\right|^2$$
$$=\sum_{\omega_{0}\in\Omega_c^{(0)}}\sum_{\stackrel{\beta_1,\dots\beta_n}{\mbox{ cycle words for $c$}}}\left|\ip{V_{\omega_{0}\beta_1\dots\beta_n}e_c}{v}\right|^2
=\sum_{\omega_{0}\in\Omega_c^{(0)}}\sum_{\beta_1,\dots,\beta_n}|\nu_{\beta_1}(c)|^2\dots|\nu_{\beta_n}(c)|^2\left|\ip{V_{\omega_0}e_c}{v}\right|^2.$$

But
$$\sum_{\beta_i\mbox{ cycle word for $c$}}|\nu_{\beta_i}(c)|^2=1,$$
by Lemma \ref{lem4.15}, so we obtain
$$\left\|P_{H_c^{(n)}}v \right\|^2=\sum_{\omega\in\Omega_c^{(0)}}\left|\ip{V_{\omega_0}e_c}{v}\right|^2=\left\|P_{H_c^{(0)}}v \right\|^2.$$
But, with (b), we obtain 
$$\left\| P_V v\right\|^2=\lim_{n\rightarrow \infty}\left\|P_{H_c^{(n)}}v \right\|^2=\left\|P_{H_c^{(0)}}v \right\|^2.$$

Next we prove the equivalences. 

(i)$\Rightarrow$(ii). If (i) holds, then $H=V$ and (ii) follows from (d). 

(ii)$\Rightarrow$(i). We know from (d) that $V$ is invariant for the representation $\{S_i\}$. Since this is irreducible, we must have $V=H$ which implies (i). 

(ii)$\Leftrightarrow$(iii). Follows from Theorem \ref{dil}.

Assume now that $\{V_i\}$ is reversing.

(i)$\Rightarrow$(iv). Follows from (e), since if $v\in K$, then $v\in H=V$. 

(iv)$\Rightarrow$(i). For this last implication we prove first the following 
\begin{lemma}\label{lem2.10}
Suppose $V$ and $K$ are closed subspaces of the Hilbert space $H$ and suppose $\{e_i \,|\, i\in I\}$ is an orthonormal basis for $V$. Then $\{P_Ke_i : i \in I\}$ is a Parseval frame for $K$ if and only if $K\subset V$.
\end{lemma}
\begin{proof}
The sufficiency is well known. Indeed, if $v\in K\subseteq V,$ then $\left\|v\right\|^2 = \sum_{i}^{} \left| \langle v, e_i \rangle \right|^2 = \sum_{i}^{} \left| \left<P_K v, e_i \right> \right|^2 = \sum_{i}^{} \left| \left<v, P_K e_i \right> \right|^2$

Assume now $\{P_K e_i : i\in I\}$ is a Parseval frame for $K$. Let $k\in K$. Then, since $\{e_i : i \in I\}$ is an orthonormal basis for $V$, we have
$$\|P_Vk\|^2=\sum_{i\in I}|\ip{k}{e_i}|^2=\sum_{i\in I}|\ip{P_Kk}{e_i}|^2=\sum_{i\in I}|\ip{k}{P_Ke_i}|^2=\|k\|^2.$$
The last equality holds because $\{P_K e_i \,|\, i\in I\}$ is a Parseval frame for $K$. But this implies that $P_Vk=k$ so $k\in V$. As $k$ was arbitrary in $K$, it follows that $K\subset V$.
\end{proof}

With Lemma \ref{l1} we have $P_KS_\omega e_c=V_\omega e_c$ for all $\omega\in\Omega_c^{(0)}$. Then, with Lemma \ref{lem2.10}, we get that $K\subseteq H_c^{(0)}\subseteq V=\Span\{S_\omega e_c : \omega\in\Omega_c^{(0)}\}$. $V$ is invariant for the representation $\{S_i\}$ therefore we obtain that $K$ is also cyclic for the representation $(V, S_i|_V)_{i=0}^{N-1}$ and clearly invariant for all $S_i^*$. Therefore $(V, S_i|_V)_{i=0}^{N-1}$ is the Cuntz dilation of the row co-isometry $(K,V_i)_{i=0}^{N-1}$. By the uniqueness in Theorem \ref{dil}, it must be isomorphic to the representation $(H, S_i)_{i=0}^{N-1}$. By (d), the Cuntz representation on $V$ is irreducible, so (ii) follows. With an earlier implication this yields (i).   

\end{proof}

\begin{proposition}\label{pr2_14}
Let $(K,V_i)_{i=0}^{N-1}$ be a row co-isometry which acts on the subspace $K_0$ as a random walk $(M,\{g_i\},\{\nu_i\})$ which is irreducible and separating. Then $\{V_i\}$ is simple on $K_0$. 
\end{proposition}

\begin{proof}
Let $T$ be as in \eqref{eqdefm1_1}. Since $K_0$ is invariant for $V_i^*$, we have $P_{K_0}V_iP_{K_0}=P_{K_0}V_i$, iterating \eqref{eqdefm1_1} we obtain, for all $n\in\bn$, 
\begin{equation}
\sum_{|\omega|=n}P_{K_0}V_\omega TV_\omega^*P_{K_0}=T
\label{eq2_14_1}
\end{equation}
Then, take $c_1\neq c_2$ in $M$. Since the random walk is separating if we take $n$ large enough, we have that, for any $\omega\in\Omega$ with $|\omega|=n$, either $\nu_\omega(c_1)=0$ or $\nu_\omega(c_2)=0$. Therefore
$$\ip{Te_{c_1}}{e_{c_2}}=\sum_{|\omega|=n}\ip{TV_\omega^* e_{c_1}}{V_\omega^*e_{c_2}}=\sum_{|\omega|=n}\nu_\omega(c_1)\cj\nu_\omega(c_2)\ip{Te_{g_\omega(c_1)}}{e_{g_\omega(c_2)}}=0.$$
So, $T$ has zero off-diagonal entries.

We prove that the diagonal entries are equal. By taking the real and imaginary parts, we can assume that $T$ is self-adjoint. Let $c_0\in M$ such that $\ip{Te_{c_0}}{e_{c_0}}=\max\{\ip{Te_c}{e_c} : c\in M\}.$

Let $c\in M$. Since the random walk is irreducible, there exists $\omega_0\in\Omega$ such that $g_{\omega_0}(c_0)=c$ and $\nu_{\omega_0}(c_0)\neq 0$. Let $n:=|\omega_0|$. We have 

$$\ip{Te_{c_0}}{e_{c_0}}=\sum_{|\omega|=n}\ip{TV_\omega^*e_{c_0}}{V_\omega^*e_{c_0}}=\sum_{|\omega|=n}|\nu_{\omega}(c_0)|^2\ip{Te_{g_\omega(c_0)}}{e_{g_{\omega}(c_0)}}$$$$\leq\sum_{|\omega|=n}|\nu_\omega(c_0)|^2\ip{Te_{c_0}}{e_{c_0}}=\ip{Te_{c_0}}{e_{c_0}}.$$
But then, we must have equalities in all inequalities, and therefore, since $\nu_{\omega_0}(c_0)\neq 0$, we get that $\ip{Te_c}{e_c}=\ip{Te_{c_0}}{e_{c_0}}$. Thus all the diagonal entries of $T$ are the same, and $T$ is a multiple of the identity, so $\{V_i\}$ is simple on $K_0$.  
\end{proof}

\begin{remark}\label{reminc}

In Proposition \ref{princ} we show that, in many cases, the family $\{S_\omega e_c : \omega\in\Omega_c^{(0)}\}$ is incomplete in $H$. 

\begin{proposition}\label{princ}
As in Theorem \ref{thm1}, let $(K,V_i)_{i=0}^{N-1}$ be a row co-isometry and let $(H,S_i)_{i=0}^{N-1}$ be its Cuntz dilation. Suppose there is a subspace $K_0$ of $K$ such that $\{V_i\}$ acts on $K_0$ as an irreducible random walk $(M,\{g_i\},\{\nu_i\})$ with orthonormal basis $(e_c)_{c\in M}$. Assume in addition that $\{V_i\}$ is simple on $K_0$.
\begin{enumerate}
	\item Suppose for all $c\in M$ and all $i\in\{0,\dots,N-1\}$, we have $|\nu_i(c)|$ is either 0 or 1. Then $M$ is a single cycle, $\{V_i\}$ is reversing, and, for $c\in M$, $n\in\bn$, and for $\omega=\omega_0\gamma\in\Omega_c^{(n)}$, with $\omega_0\in\Omega_c^{(0)}$ and $\gamma$ a loop for $c$, $S_{\omega}e_c$ is a unimodular constant multiple of $S_{\omega_0}e_c$; therefore $\{S_\omega e_c : \omega\in\Omega_c^{(0)}\}$ is an orthonormal basis for the space $V=\Span\{ S_\omega e_c : \omega\in\Omega\}$. 
\item 
Suppose that there exists $c\in M$ such that $\nu_i(c)\neq 0$ for two distinct $i\in\{0,\dots,N-1\}$. Then, for all $n\geq 0$, the family $\{S_\omega e_c : \omega\in\Omega_c^{(n)}\}$ is incomplete in $V$. 

\end{enumerate}
\end{proposition}

\begin{proof}
For (i), if the probabilities $|\nu_i(c)|^2$ are all 0 or 1, it follows that from each point in $M$ there is a unique possible transition to another point in $M$, and  since the random walk is irreducible, it follows that $M$ has to be a cycle.

We prove that $V_i$ is reversing. Indeed let $c\in M$ and let $i_0$ be the unique digit such that the transition $c\rightarrow g_{i_0}(c)$ is possible. We have 
$$e_c=\sum_{i=0}^{N-1}V_iV_i^*e_c=\sum_{i=0}^{N-1}\nu_i(c) V_ie_{g_i(c)}=\nu_{i_0}(c)V_{i_0}e_{g_{i_0}(c)},$$
so $V_{i_0}e_{g_{i_0}(c)}=\cj \nu_{i_0}(c)e_c$. Thus $\{V_i\}$ is reversing. 

Let $c\in M$, and let $\beta$ be the unique cycle word for $c$. Using the same argument as before, we get that $\{S_i\}$ is also reversing on the random walk and so $S_\beta e_c=\cj\nu_\beta(c)e_c$. But this implies that if $\omega\in\Omega_c^{(0)}$ then $S_{\omega\beta}e_c=\cj\nu_\beta(c)S_\omega e_c$. By induction, we obtain the last statements in (i).

For (ii), let $n\geq 0$. Let $\beta_0$ be a fixed cycle word for $c$. Let $\beta_0^n:=\underbrace{\beta_0\dots\beta_0}_{\mbox{$n$ times}}$. We prove that $S_{\beta_0^{n+1}}e_c$ is orthogonal to $S_\omega e_c$, for all $\omega\in\Omega_c^{(n)}$, $\omega\neq\beta_0^n$. Indeed, if this is not true, let $\omega=\omega_0\beta_1\dots\beta_n$, for some $\omega_0\in\Omega_c^{(0)}$ and $\beta_1,\dots,\beta_n$ cycle words for $c$, then by Lemma \ref{lemm2}, we have that $\beta_0^{n+1}=\omega_0\beta_1\dots\beta_n\beta_{n+1}\dots\beta_{n+p}$ for some cycle words $\beta_{n+1}\dots\beta_{n+p}$. But, from Lemma \ref{leminj}, we get that $\omega_0=\ty$ and $\beta_1=\dots=\beta_{n+p}=\beta_0$ and $p=1$. So $\omega=\beta_0^n$, a contradiction. 

Note also that, $S_{\beta_0^n}e_c$ is orthogonal to all $S_\omega e_c$, for all $\omega\in\Omega_c^{(n)}$, $\omega\neq\beta_0^n$, by Theorem \ref{thm1}. 

Consider $v=S_{\beta_0^{n+1}}e_c-\ip{S_{\beta_0^{n+1}}e_c}{ S_{\beta_0^n}e_c}S_{\beta_0^n}e_c$, which is perpendicular to $S_{\beta_0^n}e_c$, but also to all $S_\omega e_c$ with $\omega\in\Omega_c^{(n)}$, $\omega\neq \beta_0^n$. Thus $v$ is orthogonal to the subspace $H_c^{(n)}$.

We just have to check that $v$ is not zero.

If there exist $i_1,i_2\in\{0,\dots,N-1\}$, $i_1\neq i_2$, such that the transitions $c\rightarrow g_{i_1}(c)$, $c\rightarrow g_{i_2}(c)$ are possible, then, since the random walk is irreducible, there are possible transitions from $g_{i_1}(c)$ and $g_{i_2}(c)$, back to $c$. Therefore $c$ has at least two distinct cycle words $\beta_1$, $\beta_2$, and $\nu_{\beta_1}(c)\in(0,1)$. We take $c$ and $\beta_0$ in the previous argument to be $\beta_1$. Since $v$ is perpendicular to $S_{\beta_o^n}e_c$, using the Pythagorean theorem, we have:  
$$\|v\|^2=\|S_{\beta_0^{n+1}}e_c\|^2-\|\ip{S_{\beta_0^{n+1}}e_c}{ S_{\beta_0^n}e_c}S_{\beta_0^n}e_c\|^2=1-|\ip{S_{\beta_0}e_c}{e_c}|^2$$$$=1-|\ip{e_c}{S_{\beta_0}^*e_c}|^2=1-|\nu_{\beta_0}(c)|^2>0.$$

Therefore, $0\neq v \perp S_{\omega}e_c$ for all $\omega\in \Omega_c^{(n)}$ and so the family $\{S_\omega e_c: \omega\in\Omega_c^{(0)}\}$ is incomplete in $H$.

\end{proof}

\end{remark}

Let us provide an application of Theorem \ref{thm1} in the case when a row co-isometry admits periodic points. 

\begin{definition}
A non empty word $\beta$ is called irreducible if there is no word $\omega$ such that $\beta=\underbrace{\omega\omega\dots \omega}_{k \text{ times}}$ with $k\geq 2$. Let $\Omega_\beta$ be the set of words that do not end in $\beta$, including the empty word.
\end{definition}

\begin{theorem}\label{pmain}
Let $(H, S_i)_{i=0}^{N-1}$ be the Cuntz dilation of the row co-isometry $(K, V_i)_{i=0}^{N-1}$. Suppose $v_0\in K$ is a unit vector such that $V_{i_{p-1}}^*\dots V_{i_{0}}^*v_0=v_0 $ for some irreducible word $\beta=i_0\dots i_{p-1}$. Then the following are equivalent:

\begin{enumerate}
	\item $\{ V_{\omega}v_0\text{ }|\text{ }\omega \in\Omega_{\beta}\}$ is a Parseval frame.
	\item $\{ S_{\omega}v_0\text{ }|\text{ }\omega \in\Omega_{\beta}\}$ is ONB in $H$.
	\item $\{ S_{\omega}v_0\text{ }|\text{ }\omega \in\Omega\}$ spans $H$.
	\item The representation $(H, S_i)_{i=0}^{N-1}$ is irreducible.
	\item The commutant of the Cuntz dilation is trivial.
	\item$\mathcal{B}^{\sigma}(K)=\mathbb{C}\mathbf{1}_K$.	
\end{enumerate}
 
\end{theorem}

\begin{proof}
We will use the following lemma:

\begin{lemma}\label{l2} Let $S:H\to H$ with $\|S\|\leq 1$, and $v\in H$ then $Sv=v$ if and only if $S^*v=v$.
\end{lemma}
\begin{proof}  If $S^*v=v$ then, since $\|S^*\|=\|S\|\leq 1$, we have
$$0\leq \|Sv-v\|^2=\ip{Sv-v}{Sv-v}=\|Sv\|^2+\|v\|^2-\ip{v}{S^*v}-\ip{S^*v}{v}$$$$=\|Sv\|^2-\|v\|^2\leq \|v\|^2-\|v\|^2=0.$$
Therefore, $Sv=v$.

For the converse, replace $S$ by $S^*$.
\end{proof}

With the requirements of Theorem \ref{thm1} in mind we construct the following data:

\begin{equation}
M:=\{ v_0,  V_{i_{0}}^{*}v_0,  V_{i_1}^{*}V_{i_{0}}^{*}v_0, \dots ,V_{i_{p-2}}^{*} \dots V_{i_{0}}^{*}v_0  \},\quad K_0:=\cj{\text{span}} M
\label{km}
\end{equation}
and $e_c:=c$ for $c\in M$.
\begin{equation}
g_i:M\to M, \quad
g_i( v_0)=
\begin{cases} 
v_0 & \, \text{if }i\neq i_{0}\\
V_{i_{0}}^{*}v_0 & \,\text{otherwise} 
\end{cases}
,\quad g_{i}(V_{i_{k}}^{*} \dots V_{i_{0}}^{*}v_0    )=
\begin{cases}
V_{i_{k}}^{*} \dots V_{i_{0}}^{*}v_0    & \quad \text{if }i\neq i_{k+1}\\
V_{i_{k+1}}^{*}V_{i_{k}}^{*} \dots V_{i_{0}}^{*}v_0 & \quad\text{otherwise} 
\end{cases}
\label{gi}
\end{equation}

\begin{equation}
\nu_i: M\to\mathbb{C}, \quad 
\nu_i (v_0)= 
\begin{cases} 
0 & \quad \text{if }i\neq i_{0}\\
1 & \quad\text{otherwise} 
\end{cases}
,\quad 
\nu_i ( V_{i_{k}}^{*} \dots V_{i_{0}}^{*}v_0    )=
\begin{cases}
0 & \quad \text{if }i\neq i_{k+1}\\
1 & \quad\text{otherwise} 
\end{cases}
\label{nui}
\end{equation}

Note that the maps $g_i$ are trivially one-to-one when the transitions are possible. 

Notice that for words $\omega$ containing a digit $i\notin\{i_0,i_1,\dots , i_{p-1}\}$ the transitions $c\to g_{\omega}(c)$ are not possible because $\nu_{\omega}(c)=0$ in this case. 
Because $v_0\in K$ we see that
\[
    M=\{    v_0,  S_{i_{0}}^{*}v_0,  S_{i_1}^{*}S_{i_{0}}^{*}v_0,\dots , S_{i_{p-2}}^{*} \dots S_{i_{0}}^{*}v_0     \}.
\]
\begin{claim}
$M$ forms an orthonormal basis for $K_0$.
\end{claim}
First we notice the elements in $M$ are unit vectors. 
Because $\left\|S_k^*\right\|=1$ for all $0\leq k\leq p-1$,  we have:
$$1=\left\|v_0\right\|=\left\| S_{i_{p-1}}^*\dots S_{i_{0}}^*v_0  \right\| \leq  \left\| S_{i_{k}}^*\dots S_{i_{0}}^*v_0 \right\| \leq  \left\| S_{i_{k}}^*\dots S_{i_{0}}^*\right\|\left\|v_0\right\|\leq 1$$
To show, for $k\neq l$
\begin{equation}\label{adperp}
\ip{ S_{i_{k}}^*\dots S_{i_0}^* v_0   }{  S_{i_{l}}^*\dots S_{i_0}^* v_0   }=0
\end{equation}
assume without loss of generality $k < l$ where $0\leq k, l\leq p-1$. \\
With  Lemma \ref{l2}, $S_{i_{p-1}}^*\dots S_{i_{0}}^*v_0=v_0 $ entails  $S_{\beta}v_0=v_0$. 
Using the Cuntz orthogonality relations $S_i^*S_i= I $ and $v_0=S_{\beta}v_0=S_{i_0}S_{i_1}\dots S_{i_{p-1}}v_0$, the left-hand side in \eqref{adperp} can be rewritten as $  \ip{ S_{i_{k+1}}\dots S_{i_{p-1}} v_0  }{    S_{i_{l+1}}\dots S_{i_{p-1}} v_0  }$. 
If this last term would not vanish then with  Lemma \ref{lemm2} we get that $\gamma:= i_j\dots i_{p-1}$ is a loop for $v_0$, for some $j> k$. From \eqref{gi} we have to have $S_{i_{p-1}}^*\dots S_{i_{j}}^*v_0=v_0$, which then gives $S_{i_{j}}\dots S_{i_{p-1}}v_0=v_0$  with Lemma \ref{l2}. Exploiting these fixed point properties for $S_{\gamma}$ and $S_{\beta}$ we have succesively in finitely many steps  
$$\ip{v_0}{v_0}=\ip{S_{\beta\beta\dots \beta}v_0}{S_{\gamma\beta\beta\dots \beta}v_0}.$$
Since the isometries $(S_l)_l$ have orthogonal ranges we have that the infinite words are equal: \\
$\beta\beta\dots =\gamma\beta\beta\dots $ with $\gamma$ a word shorter than $\beta$.  Then 
$\gamma (\beta\beta\dots )=\gamma\gamma\beta\beta\dots =\dots =\gamma\gamma\dots \gamma\beta\beta\dots $, which implies that 
$$\gamma\gamma\gamma\dots=\beta\beta\beta\dots ,\quad\text{ with }q:=|\gamma|<|\beta|=p.$$
Let $d=\text{gcd }(p,q)<p$. Then there exists $m$, $n$ in $\mathbb{Z}$ such that $mp+nq=d$. 
Denote $\gamma_k=\gamma_{k\text{ mod  }q}$ and $\beta_k=i_{k\text{ mod  }p}$ for all $k\in\mathbb{Z}$. Then from the last equality above with infinite words, we have for $\gamma_k=\beta_k$ for all $k\in\mathbb{Z}$. Then for all $k$: 
$$\beta_{d+k}=\beta_{mp+nq+k}=\beta_{nq+k}=\gamma_{nq+k}=\gamma_k=\beta_k.$$
 So $\beta\beta\dots $ actually has period $d$, i.e. $\beta=\underbrace{\alpha\alpha\dots \alpha}_{p/d \text{ times}}$ where $\alpha=i_0\dots i_{d-1}$. This contradicts the fact that $\beta$ is irreducible. 
Hence \eqref{adperp} holds and the Claim is proved.

Next, we check that the notions introduced in Definition \ref{defm1} are satisfied in this setting. 
\begin{claim} $V_i^{*}K_0\subset K_0$ for all $i\in\{0,\dots , N-1\}$. Thus item (i) in Definition \ref{defm1} holds.
\end{claim}
From the definition of $M$ above we see that $V_{i_k}^{*}( c )\in M$ if  $c=V_{i_{k-1}}^*\dots V_{i_0}^* v_0$. For the remaining cases we prove that $V_i^{*}(c)=0$, thus finishing the Claim. \\
 $S_{i_{p-1} }^*\dots S_{i_{0} }^* v_0=v_0$. 
Let $0\leq k\leq p-1$. Then $S_{i_0}\dots S_{i_k}S_{i_k}^*\dots S_{i_0}^*$ is the projection onto $S_{i_0}\dots S_{i_k} H$, and is smaller than the projection $S_{i_0}\dots S_{i_{k-1}}S_{i_{k-1}}^*\dots S_{i_0}^*$. Hence
$$\| v_0 \|^2=\|  S_{i_0}\dots S_{i_{p-1}}S_{i_{p-1}}^*\dots S_{i_0}^* v_0  \|^2\leq \|     S_{i_0}\dots S_{i_k}S_{i_k}^*\dots S_{i_0}^* v_0       \|^2 \leq \|v_0\|^2$$
which implies $\| v_0 \|^2=\|     S_{i_0}\dots S_{i_k}S_{i_k}^*\dots S_{i_0}^* v_0       \|^2$. Since $S_{i_0}\dots S_{i_k}S_{i_k}^*\dots S_{i_0}^*$ is a projection, it follows that  

\begin{equation}\label{mpfix}
v_0 = S_{i_0}\dots S_{i_k}S_{i_k}^*\dots S_{i_0}^* v_0.   
\end{equation}

Then using the Cuntz relations 
$$\sum_{|\omega|=k+1}\ip{S_{\omega}S_{\omega}^*v_0 }{v_0}=\ip{v_0}{v_0} =\ip{  S_{i_0\dots i_k}S_{i_k}^*\dots S_{i_0}^* v_0    }{v_0}.$$
It follows that $\ip{S_{\omega}S_{\omega}^*v_0 }{v_0}=0$ $\forall$ $\omega\neq i_0\dots i_k$, $|\omega|=k+1$. Hence
\begin{equation}\label{szero}
V_{\omega}^{*}v_0=S_{\omega}^*v_0=0\quad\forall \text{ }\omega\neq i_0\dots i_k,\text{ } |\omega|=k+1.
\end{equation}
This concludes the Claim. \\
Notice that item (ii) in Definition \ref{defm1} is satisfied because of equation \eqref{szero} and the choice of $M$, $g_i$ and $\nu_i$. Also the random walk $(M, \{g_i\}, \{ \nu_i\}  )$ is irreducible: if $c_1=V_{i_{k}}^{*} \dots V_{i_{0}}^{*}v_0\in M$ and $c_2=V_{i_{l}}^{*} \dots V_{i_{0}}^{*}v_0\in M$ with say $k<l$ then $g_{\omega}(c_1)=c_2$ and $\nu_{\omega}(c_1)=1 $, where 
$\omega=i_{k+1}\dots i_l$.  
\begin{claim}
$\{V_i\}$ is {\it{reversing}} on the random walk $(M, \{g_i\}, \{ \nu_i\}  )$.   
\end{claim}
Suppose $\nu_i(c)\neq 0$ for $c\in M$.  We need show $V_i(g_i(c))=\cj{\nu_i(c) }c$. Because $c\in M$ we have $c=V_{i_{k}}^*\dots V_{i_0}^* v_0$ for some $0\leq k\leq p-1$. 
 From the definition of the $\nu_i$ we must have $i=i_{k+1}$ (the case $c=v_0$ follows similarly). What is needed to show now is 
\begin{equation}\label{revs}
V_{i_{k+1}}(  V_{i_{k+1}}^*\dots V_{i_0}^* v_0   )= V_{i_{k}}^*\dots V_{i_0}^* v_0.
\end{equation}
 From \eqref{mpfix} with $k\to k+1$ and rewriting using the Cuntz relations $S_i^*S_i=I$ we have
$$S_{i_k}^*\dots S_{i_0}^* v_0= S_{i_{k+1}}(  S_{i_{k+1}}^*\dots S_{i_0}^* v_0     ).    $$
We apply the projection $P_K$ on both sides and use $v_0\in K$ and Lemma \ref{l1} to obtain \eqref{revs}.  
\begin{claim} The random walk $(M, \{g_i\}, \{ \nu_i\}  )$ is separating. 
\end{claim}
Let $c_1:=V_{i_{k}}^{*} \dots V_{i_{0}}^{*}v_0$ and $c_2:=V_{i_{l}}^{*} \dots V_{i_{0}}^{*}v_0$ in $M$, with $0\leq k< l \leq p-1$. 
We prove that if $\omega$ is a word with $|\omega|\geq p$ then either $\nu_{\omega}(c_1)=0$ or $\nu_{\omega}(c_2)=0$. Suppose by contradiction that for some $\omega=j_0j_1\dots j_t$ with $t\geq p$ both $\nu_{\omega}(c_1)$ and $\nu_{\omega}(c_2)$ are non zero. Using \eqref{nui} we have $j_0=i_{k+1}=i_{l+1}$, $j_1=i_{k+2}=i_{l+2}$, and continuing we get $i_{k+t \text{ mod }p}=i_{l+t \text{ mod }p} $ for all $t\in\mathbb{N}$. As a consequence the following infinite words are equal: 
$$i_{k+1}\dots i_{p-1}\beta\beta\dots = i_{l+1}\dots i_{p-1}\beta\beta\dots $$

Concatenating the word $i_0i_1\dots i_k$ on the left hand side we  obtain
$$\beta\beta\beta\dots =\gamma \beta\beta \dots \quad \text{where }\gamma=i_0i_1..i_ki_{l+1}\dots i_{p-1}.$$
Using the same argument as above we get a contradiction with the fact that $\beta$ is irreducible.

With Proposition \ref{pr2_14}, it follows that $\{V_i\}$ is simple on $K_0$.

We are done checking the settings and requirements of Theorem \ref{thm1}. We apply its item (d) and finish the proof. 
\end{proof}

In the particular case when one of the co-isometries admits a fixed point one easily obtains the following characterizations. 

\begin{corollary}\label{main}
Let $(H, S_i)_{i=0}^{N-1}$ be the Cuntz dilation of row co-isometry $(K, V_i)_{i=0}^{N-1}$, and $v_0\in K$ unit vector such that $V^*_0(v_0)=v_0$. The following are equivalent:
\begin{enumerate}
	\item The family of vectors $\{V_{\omega}v_0\}_{\omega\in\Omega_0}$ is a Parseval frame in $K$.
	\item The family of vectors $\{S_{\omega}v_0\}_{\omega\in\Omega_0}$ is ONB in $H$.
		\item The family of vectors $\{S_{\omega}v_0\}_{\omega\in\Omega_0}$ spans $H$.
	\item The representation $(H, S_i)_{i=0}^{N-1}$ is irreducible.
	\item The commutant of the Cuntz dilation is trivial.
	\item  $\mathcal{B}^{\sigma}(K)=\mathbb{C}\mathbf{1}_K$.
\end{enumerate}

\end{corollary}

\begin{example}
If the Cuntz dilation is not irreducible, it is possible that $\{V_\omega v_0\}_{\omega\in\Omega_0}$ is not even a frame for its span.

Define the operators $V_0^*,V_1^*$ on $l^2(\mathbb Q)$ by
$$V_0^*e_r=\lambda_r^0e_{r/2},\quad V_1^*e_r=\lambda_r^1e_{(r-1)/2},\quad (r\in\mathbb Q),$$
where $e_r$ is the canonical basis in $l^2(\mathbb Q)$, $e_r(q)=\delta_{r,q}$, $r,q\in\mathbb Q$, and
$\lambda_0^0=1$, $\lambda_0^1=0$, and $\lambda_r^i=1/\sqrt2$ for $r\neq 0$, $i=0,1$. Thus $|\lambda_r^0|^2+|\lambda_r^1|^2=1$ for all $r$.

Denote
$$f_0(r)=\frac{r}2,\quad f_1(r)=\frac{r-1}2,\quad (r\in\mathbb Q),$$
with inverses
$$h_0(r)=2r, \quad h_1(r)=2r+1,\quad (r\in\mathbb Q).$$

Clearly $V_i^*$ defines a bounded operator on $l^2(\mathbb Q)$. We compute the adjoints

$$\ip{V_ie_r}{e_q}=\ip{e_r}{V_i^*e_q}=\cj\lambda_q^i\delta_{rf_i(q)}=\cj\lambda_q^i\delta_{h_i(r)q}=\cj\lambda_{h_i(r)}^i\delta_{h_i(r)q}=\ip{\cj\lambda_{h_i(r)}^ie_{h_i(r)}}{e_q}.$$
Therefore
$$V_ie_r=\cj\lambda_{h_i(r)}^ie_{h_i(r)},\quad (r\in\mathbb Q,i=0,1).$$
Then, we have
$$\sum_{i=0}^1V_iV_i^*e_r=\sum_{i=0}^1\lambda_r^iV_ie_{f_i(r)}=\sum_{i=0}^1\lambda_r^i \cj\lambda_{h_i(f_i(r))}^i e_{h_i(f_i(r))}=\sum_{i=0}^1|\lambda_r^i|^2e_r=e_r.$$
Thus $V_0V_0^*+V_1V_1^*=I.$

Also, if $v_0=e_0$ then $V_0^*v_0=v_0$.

Now, for $i_1,\dots,i_n\in\{0,1\}$, we have

$$V_{i_1}\dots V_{i_n}(V_1v_0)=V_{i_1}\dots V_{i_n}(\cj\lambda_{1}^1e_{1})=V_{i_1}\dots V_{i_{n-1}}(\cj\lambda_{h_{i_n}(1)}^1\cj\lambda_1^{1} e_{h_{i_n(1)}})$$$$=\dots
=\cj\lambda_{h_{i_1}\dots h_{i_n}(1)}^{i_1}\cj\lambda_{h_{i_2}\dots h_{i_n}(1)}^{i_2}\dots \cj\lambda_{h_{i_n}(1)}^{i_n}\cj\lambda_1^1 e_{h_{i_1}\dots h_{i_n}(1)}=\left(\frac1{\sqrt2}\right)^{n+1}e_{2^n+2^{n-1}i_n+2^{n-2}+\dots+i_1}.$$

Consider the vector $e_{2^m}$, $m\geq 1$. If $2^m=2^n+2^{n-1}i_n+2^{n-2}+\dots+i_1$, then $n=m$ and $i_n=\dots=i_1=0$, and
$$\ip{e_{2^m}}{V_{i_1}\dots V_{i_n}(V_1e_0)}=\left(\frac1{\sqrt2}\right)^{m+1}.$$

If $2^m\neq 2^n+2^{n-1}i_n+2^{n-2}+\dots+i_1$, then

$$\ip{e_{2^m}}{V_{i_1}\dots V_{i_n}(V_1e_0)}=0.$$

If $\{V_\omega v_0\}_{\omega\in\Omega_0}$ is a frame,then for a frame constant $A>0$, we have
$$A=A\|e_{2^m}\|^2\leq \sum_\omega |\ip{e_{2^m}}{V_\omega v_0}|^2=\left[\left(\frac1{\sqrt2}\right)^{m+1}\right]^2.$$ Letting $m\rightarrow\infty$, we get $A=0$, a contradiction.

\end{example}

\begin{proof}[Proof of Theorem \ref{th3.5}]
We remark first the following: suppose $M_1,M_2$ are two different minimal compact invariant sets. Then they have to be disjoint (see Lemma \ref{lem4.11-0}). Indeed, if they are not, then $M_1\cap M_2$ is compact invariant set which is contained in $M_1$, and, since $M_1$ is minimal, $M_1\cap M_2=M_1$, and similarly for $M_2$, which is a contradiction.

We also note, that since the functions $t\mapsto V_i^*e_t$ and $g_i(t)$ are continuous, and using relation \eqref{eqas2}, we get 
$$\nu_i(t)=\ip{\nu_i(t)e_{g_i(t)}}{e_{g_i(t)}}=\ip{V_i^*e_t}{e_{g_i(t)}},$$
it follows that the maps $\nu_i(t)$ are also continuous.

\begin{lemma}\label{lem4.11-0}
For $a\in \mathcal{T}$, the closed orbit $\cj{\mathcal O(a)}$ is compact and invariant. 
If $M$ is a minimal compact invariant set, and $x_0\in M$, then $M=\cj{\mathcal O(x_0)}$. If $M_1, M_2$ are distinct compact minimal invariant sets, they are disjoint. 

\end{lemma}

\begin{proof} First, we prove that the entire orbit of $a$, that is $O:=\cj{\{g_\omega(a) : \omega\in\Omega\}}$ is compact. Let $c\in (0,1)$ be a common contraction constant for all $(g_i)$, $i=0,\dots ,N-1$.  
It suffices to show that any sequence in $O$ has a convergent subsequence. Now let $(g_{\omega_n}(a))_{n\geq 1}$, $\omega_n\in\Omega$, be a sequence in the orbit of $a$ and fix some $x_0\in X_0$. By passing to a subsequence, we may assume that the sequence of lengths $|\omega_n|\to \infty$, otherwise the orbit would be finite, thus one could select a constant subsequence.  

Let $X_0$ be the unique compact attractor of the iterated function sytem $(g_i)_{i=0}^{N-1}$, i.e., the unique compact subset with the property that 
$$X_0=\bigcup_{i=0}^{N-1}g_i(X_0).$$
See \cite[Section 3]{Hut81} or \cite{MR3838440} for details.
 
Because the attractor $X_0$ is compact, one can extract a convergent subsequence of $(g_{\omega_n}(x_0))_{n\geq 1}$. There is no loss of generality if we relabel the subsequence and still denote it by $(g_{\omega_n}(x_0))_{n\geq 1}$. Say $g_{\omega_n}(x_0)\to x\in X_0$. Let $\epsilon>0$ and $N$ large enough so that $c^{|{\omega_n}|}d(a,x_0)<\epsilon/2$ and $d(g_{\omega_n}(x_0),x)<\epsilon/2$ whenever $n>N$. Using the triangle inequality, we have, for $n>N$:
$$d(g_{\omega_n}(a), x)\leq d(g_{\omega_n}(a), g_{\omega_n}(x_0)  )+d(g_{\omega_n}(x_0),x)\leq c^{|\omega_n|}d(a,x_0)+d(g_{\omega_n}(x_0),x)<\epsilon $$
Hence the subsequence $(g_{\omega_n}(a))_{n\geq 1}$ is convergent.

 Next we show that $\cj{\mathcal O(a)}$ is invariant. Note first that $\mathcal{ O}(a)$ is invariant: indeed if $b\in \mathcal{O}(a)$ with $b\to g_i(b)$ possible then $b=g_{\omega}(a)$ with $a\to b$ possible. So $\nu_i(b)\neq 0$ and $\nu_{\omega i}(a)=\nu_{\omega}(a)\nu_i(b)\neq 0$. It follows that $a\to g_{\omega i}(a)=g_i(b)$ is possible, hence $g_i(b)\in \mathcal{O}(a)$.  Now, if $b\in   \cj{\mathcal O(a)}$ with $b\to g_i(b)$ possible (so $\nu_i(b)\neq 0$) then one can find a sequence $ b_n\in \mathcal O(a)$ with $\displaystyle{\lim_{n\to \infty} b_n=b}$. Hence the transitions $a\to b_n=g_{\omega_n}(a)$ are possible, where $\omega_n$ are words depending on $n$. Because $g_i$ and $\nu_i$ are continuous we obtain $\displaystyle{g_i(b)=\lim_{n\to \infty}g_i(b_n)}$ and $\nu_i(b_n)\neq 0$ for $n$ large enough. It follows that the transitions 
$a\to g_{\omega_n i}(a)=g_i(b_n)$ are possible, so $g_i(b_n)\in \mathcal{O}(a)$. Because $g_i(b)=\displaystyle{\lim_{n\to \infty}g_i(b_n)}$ we obtain $g_i(b)\in     \cj{\mathcal O(a)}$.  

If $x_0\in M$ then $\mathcal O(x_0)$ is contained in $M$. Since $M$ is minimal and $\cj{\mathcal O(x_0)}$ is invariant, we get $M=\cj{\mathcal O(x_0)}$.

If $M_1, M_2$ are distinct compact minimal invariant sets, and we assume that they are not disjoint, then for $x_0\in M_1\cap M_2$ we have $M_1=\cj{\mathcal O(x_0)}=M_2$, a contradiction. 

\end{proof}

%If, in addition the space $\mathcal T$ is a complete metric space and the maps $\{g_i\}$ are contractions, we claim that the minimal compact invariant sets are contained in the attractor of the iterated function system $\{g_i\}$, that is the unique compact set $X_0$ in $\mathcal T$ with the property that 

\begin{lemma}\label{lemx0}
The minimal compact invariant sets are contained in the attractor $X_0$ of the iterated function system $\{g_i\}$.
\end{lemma}

\begin{proof}
Indeed, let $M$ be a minimal compact invariant set. Let $x\in M$. Pick $\omega_1,\omega_2,\dots$ such that the transitions $x\rightarrow g_{\omega_1}(x)\rightarrow g_{\omega_1\omega_2}(x)\rightarrow\dots$ are possible. Then $g_{\omega_1\dots\omega_n}(x)\in M$, for all $n\geq 0$. Let $x_0\in X_0$. Then, if $0<c<1$ is a common contraction ratio for the maps $\{g_i\}$, we have $d(g_{\omega_1\dots\omega_n}(x),g_{\omega_1\dots\omega_n}(x_0))\leq c^n  d(x,x_0)\rightarrow 0$. Since $g_{\omega_1\dots\omega_n}(x_0)\in X_0$ and $X_0$ is compact, it follows that is has a convergent subsequence to a point $z_0$ in $X_0$. Then the same subsequence of $\{g_{\omega_1\dots\omega_n}(x)\}$ is convergent to $z_0$, which means $z_0\in M$. Then $\cj{\mathcal O(z_0)}$ is a compact invariant set contained in $M$ and in $X_0$. Since $M$ is minimal $M=\cj{\mathcal O(z_0)}\subseteq X_0$. 

Thus, all minimal compact invariant sets are contained in the compact attractor of the iterated function system $\{g_i\}$. 
\end{proof}

\begin{lemma}\label{lemdist}
There exists a constant $\delta>0$ such that for any two distinct minimal compact invariant sets $M_1$ and $M_2$, $\dist(M_1,M_2)\geq \delta$.

\end{lemma}

\begin{proof}
We follow the argument in \cite{CCR96}. Since the maps $|\nu_i|^2$ are uniformly continuous on the attractor $X_0$, there exists $\delta>0$ such that, if $d(x,y)<\delta$, then $||\nu_i(x)|^2-|\nu_i(y)|^2|<\frac1N$ for all $i$. 

Suppose now, that there exist $x\in M_1$ and $y\in M_2$ such that $d(x,y)<\delta$. Since $\sum_i |\nu_i(x)|^2=1$, it follows that there exist $i_1$ such that $|\nu_{i_1}(x)|^2\geq\frac1N$. Then, since $d(x,y)<\delta$, we obtain that $|\nu_{i_1}(y)|^2>0$. This means that the transitions $x\rightarrow g_{i_1}(x)$ and $y\rightarrow g_{i_1}(y)$ are possible, so $g_{i_1}(x)\in M_1$ and $g_{i_1}(y)\in M_2$. Also $d(g_{i_1}(x),g_{i_1}(y))\leq d(x,y)<\delta$. Repeating the argument and using induction, we obtain some labels $i_1,i_2,\dots,i_n,\dots$ such that the transitions $x\rightarrow g_{i_1\dots i_n}(x)$ and $y\rightarrow g_{i_1\dots i_n}(y)$ are possible for all $n$. So $g_{i_1\dots i_n}(x)\in M_1$ and $g_{i_1\dots i_n}(y)\in M_2$. But, using the same argument as in Lemma \ref{lemdist}, $d(g_{i_1\dots i_n}(x), g_{i_1\dots i_n}(y))$ converges to 0, which would mean that $\dist(M_1,M_2)=0$ and this contradicts the fact that $M_1$ and $M_2$ are distinct minimal compact invariant sets. 
\end{proof}

Using Lemma \ref{lemdist} and Lemma \ref{lemx0}, all minimal compact invariant sets are cpntained in the compact attractor $X_0$ and they have a distance bigger than $\delta$ between any two of them. Thus, there can be only a finite number of minimal compact invariant sets.

We start with a simple relation that we will use often. For a word $\omega\in\Omega$ and $t\in\mathcal T$,
$$S_{\omega}^*e_t=V_\omega^*e_t=\nu_\omega(t)e_{g_{\omega}(t)}.$$
We have, for $t\in M_j$ and $i\in\{0,\dots,N-1\}$,
 $V_i^*e_t=\nu_i(t)e_{g_i(t)}\in K(M_j)$, if the transition $t\rightarrow g_i(t)$ is possible, and $V_i^*e_t=0\in K(M_j)$ if the same transition is not possible. So $K(M_j)$ is invariant for $V_i^*$.

The next lemma shows that the spaces $H(M_i)$ are invariant for the Cuntz representation $\{S_j\}$. 
\begin{lemma}\label{lem3.5}
If $K_0$ is a subspace of $K$ which is invariant under the operators $V_i^*$, $i=0,\dots, N-1$, then
$$H(K_0):=\Span\{S_\omega v : v\in K_0, \omega\in\Omega\},$$
is invariant for the representation $(S_i)_{i=0}^{N-1}$.
\end{lemma}

\begin{proof} Clearly, the space is invariant under the operators $S_i$, $i=0,1,\dots,N-1$.
For a nonempty word $\omega=\omega_1\dots \omega_n$, $v\in K_0$ and $i\in\{0,\dots,N-1\}$ we have $S_i^*(S_\omega v)=\delta_{i,\omega_1}S_{\omega_2}\dots S_{\omega_n}v\in H(K_0)$. Also $S_i^*v=V_i^*v\in K_0\subseteq H(K_0)$.
\end{proof}

Next, we prove that the spaces $K(M_j)$ are mutually orthogonal.

\begin{lemma}\label{lem3.6}
If $M_1$ and $M_2$ are two closed invariant subsets with $\textup{dist}(M_1,M_2)>0$ then
$$\ip{e_{t_1}}{e_{t_2}}=0,\quad (t_1\in M_1, t_2\in M_2).$$
\end{lemma}

\begin{proof}
Let $\omega$ be a word and assume that the transitions $t_1\rightarrow g_\omega(t_1)$ and $t_2\rightarrow g_\omega(t_2)$ are possible in several steps. Then
$g_\omega(t_1)\in M_1$ and $g_\omega(t_2)\in M_2$. Also, since the maps $g_i$ are contractions with ratio $0<c<1$, if $n=|\omega|$ then
$$\textup{dist}(g_\omega(t_1), g_\omega(t_2))\leq c^n \textup{dist}(t_1,t_2),$$
so, if $n$ is large enough then
$$\textup{dist}(g_\omega(t_1), g_\omega(t_2))<\textup{dist}(M_1,M_2),$$
and this would yield a contradiction. Thus, for $\omega$ long enough, one of the transitions $t_1\rightarrow g_\omega(t_1)$ or $t_2\rightarrow g_\omega(t_2)$ is not possible, so $\nu_\omega(t_1)=0$ or $\nu_\omega(t_2)=0$.

Using this, with $n$ large enough, and the co-isometry relation, we have
$$\ip{e_{t_1}}{e_{t_2}}=\sum_{|\omega|=n}\ip{S_\omega^*e_{t_1}}{S_\omega^*e_{t_2}}=\sum_{|\omega|=n}\nu_\omega(t_1)\cj \nu_\omega(t_2)\ip{e_{g_\omega(t_1)}}{e_{g_\omega(t_2)}}=0.$$
\end{proof}

\begin{lemma}\label{lem3.7}
Assume that the subspaces $K_1$ and $K_2$ are invariant for the operators $V_i^*$, $i=0,\dots,N-1$, and are mutually orthogonal, and let $H(K_1)$ and $H(K_2)$ be the subspaces of $H$ as in Lemma \ref{lem3.5}. Then $H(K_1)$ and $H(K_2)$ are orthogonal.
\end{lemma}

\begin{proof}
Given two words $\omega=\omega_1\dots\omega_n$ and $\omega'=\omega_1'\dots\omega_m'$, if there exists $i$ such that $\omega_i\neq \omega_i'$ then, take the first such $i$ and, for $v_1\in K_1$ and $v_2\in K_2$:
$$\ip{S_\omega v_1}{S_{\omega'}v_2}=\ip{S_{\omega_i}S_{\omega_{i+1}}\dots S_{\omega_n}v_1}{S_{\omega_i'}S_{\omega_{i+1}'}\dots S_{\omega_m'}v_2}=0,$$
since the ranges of $S_{\omega_i}$ and $S_{\omega_i'}$ are orthogonal.

In the remaining case, $\omega'$ is a prefix of $\omega$ (or vice-versa), $\omega=\omega'\beta$ for some word $\beta_1\dots\beta_p$, then
$$\ip{S_\omega v_1}{S_{\omega'}v_2}=\ip{S_\beta v_1}{v_2}=\ip{v_1}{S_\beta^*v_2}=0,$$
because $S_\beta^*v_2\in K_2$.
\end{proof}

Using Lemma \ref{lem3.6} and Lemma \ref{lem3.7}, we get the next lemma:

\begin{lemma}\label{lem4.6.1}
If $M_1$ and $M_2$ are two closed invariant subsets with $\textup{dist}(M_1,M_2)>0$ then the  subspaces $K(M_1)$ and $K(M_2)$ are orthogonal and the subspaces $H(M_1)$ and $H(M_2)$ are orthogonal.
\end{lemma}

Thus, the spaces $H(M_j)$ are mutually orthogonal.

To prove the irreducibility of the Cuntz representation on $H(M_j)$, the disjointness and the fact that the sum of these spaces is $H$, we will introduce the Ruelle operator:

\begin{definition}\label{defas1}
 Define the Ruelle transfer operator for functions $f:\mathcal T\rightarrow\bc$ by
\begin{equation}
Rf(t)=\sum_{i=0}^{N-1}|\nu_i(t)|^2f(g_i(t)),\quad (t\in\mathcal T).
\label{eqruelle}
\end{equation}

Recall that $\mathcal B(K)^\sigma$ is the space of bounded operators $T$ on $K$ with 
$$\sum_{i=0}^{N-1}V_iTV_i^*=T.$$
By Theorem \ref{dil}, the commutant of the representation $(H,S_i)_{i=0}^{N-1}$ is in bijective correspondence with $\mathcal B(K)^\sigma$ by the map $A\mapsto P_KAP_K$. 

Define the map $A\mapsto h_A$, from $\mathcal B(K)^\sigma$ to complex valued functions defined on $\mathcal T$, by
\begin{equation}
h_A(t)=\ip{Ae_t}{e_t},\quad (t\in\mathcal T).
\label{eqha}
\end{equation}

For an operator $A$ in the commutant of the representation $(H,S_i)_{i=0}^{N-1}$, the operator $P_KAP_K$ is in $\mathcal B(K)^\sigma$, and we define $h_A=h_{P_KAP_K}$.

\end{definition}

\begin{lemma}\label{las1}

The function $\mathcal B(K)^\sigma\ni A\mapsto h_A$, maps into the continuous fixed points of the Ruelle transfer operator, is linear, and order preserving. Also, $R1=1$.

\end{lemma}

\begin{proof}
It is clear that, for $A\in\mathcal B(K)^\sigma$, the function $h_A$ is continuous. Also, the map $A\mapsto h_A$ is clearly linear and order preserving. We check that $h_A$ is a fixed point for $R$. We have, using the assumptions \eqref{eqas1}-\eqref{eqas2},
$$h_A(t)=\ip{Ae_t}{e_t}=\sum_{i=0}^{N-1}\ip{AV_i^*e_t}{V_i^*e_t}=\sum_{i=0}^{N-1}|\nu_i(t)|^2\ip{Ae_{g_i(t)}}{e_{g_i(t)}}$$$$=\sum_{i=0}^{N-1}|\nu_i(t)|^2h_A(g_i(t))=(Rh_A)(t).$$
So $h_A$ is a fixed point for $R$. In particular, taking $A$ to be the identity, since $h_I(t)=\|e_t\|^2=1$ for all $t$, it follows that $R1=1$.

\end{proof}

\begin{lemma}\label{lem4.7.1} 
Let $t_1,t_2\in\mathcal T$ and $\epsilon>0$. Then there exists $n\in\bn$ such that if $\omega\in \Omega$ has length $|\omega|\geq n$, then $$\|e_{g_\omega(t_1)}-e_{g_\omega(t_2)}\|<\epsilon.$$

\end{lemma}

\begin{proof}
Let $X_0$ be the compact attractor of the iterated function system $(g_i)_{i=0}^{N-1}$. We claim that there is a $\delta>0$ such that, if $x_0\in X_0$, $t\in\mathcal T$ and $\dist(x_0,t)<\delta$, then $\|e_{x_0}-e_t\|<\epsilon/2$.

Suppose not. Then there exists a sequence $\{x_n\}$ in $X_0$ and a sequence $\{t_n\}$ in $\mathcal T$, such that $\dist(x_n,t_n)<1/n$ and $\|e_{x_n}-e_{t_n}\|\geq \epsilon/2$. Since $X_0$ is compact, by passing to a subsequence, we may assume that $\{x_n\}$ converges to some $x_0$ in $X_0$. Since $\dist(x_n,t_n)$ converges to 0, we get that $t_n$ also converges to $x_0$. Since the map $t\mapsto e_t$ is continuous, we get that both $e_{x_n}$ and $e_{t_n}$ converge to $e_{x_0}$, a contradiction.

Now, take some point in $x_0$ in $X_0$. Since the maps $g_i$ are contractions, with some contraction ratio $0\leq c<1$, if we take $n$ large enough and $\omega$ in $\Omega$ with length $|\omega|=n$, we have, for $i=1,2$:
$$\dist(g_\omega (t_i),g_\omega (x_0))\leq c^n d(t_i,x_0)<\delta,$$
and since $g_\omega (x_0)$ is in $X_0$, we get that 
$\|e_{g_\omega (t_i)}-e_{g_\omega (x_0)}\|<\epsilon/2$, which implies that $\|e_{g_\omega (t_1)}-e_{g_\omega (t_2)}\|<\epsilon$.

\end{proof}

\begin{lemma}\label{lem4.8}
Let $A$ be in the commutant of $\{S_i\}$. If $h_A=0$ then $A=0$. 
\end{lemma}
%
%\begin{proof}
%We have, for $t_1,t_2\in \mathcal T$
%$$\ip{Ae_{t_1}}{e_{t_2}}=\sum_{|\omega|=n}\ip{S_\omega^*Ae_t}{S_\omega^*e_t}=\sum_{|\omega|=n}\nu_\omega(t_1)\cj\nu_\omega(t_2)\ip{Ae_{g_\omega t_1}}{e_{g_\omega t_2}}.$$
%But, for $n$ large, $g_\omega t_1$ and $g_\omega t_2$ are close, therefore $|\ip{Ae_{g_\omega t_1}}{e_{g_\omega t_2}}|\leq \epsilon$ uniformly in $\omega$ (we might need some compactness here). Then, with Cauchy-Schwarz
%$$|\ip{Ae_{t_1}}{e_{t_2}}|\leq \left(\sum_{|\omega|=n}|\nu_\omega(t_1)|^2\right)^{\frac12}\left(\sum_{|\omega|=n}|\nu_\omega(t_2)|^2\epsilon^2\right)^{\frac12}=\epsilon.$$
%
%Since $\epsilon$ was arbitrary, we get $\ip{Ae_{t_1}}{e_{t_2}}=0$ for all $t_1,t_2$, and since these vectors span the entire space, we get $A=0$.
%\end{proof}
%
%Comment: (alternative) proof, we need uniform continuity of the map $t\to e_t$ again:
%

\begin{proof}
We have, for $t_1,t_2\in \mathcal T$
$$\ip{Ae_{t_1}}{e_{t_2}}=\sum_{|\omega|=n}\ip{S_\omega^*Ae_{t_1}}{S_\omega^*e_{t_2}}=\sum_{|\omega|=n}\nu_\omega(t_1)\cj\nu_\omega(t_2)\ip{Ae_{g_\omega (t_1)}}{e_{g_\omega (t_2)}}.$$
Suppose $A\neq 0$. By Lemma \ref{lem4.7.1}, for a given $\epsilon>0$, for $\omega$ long enough, 
 $\|     e_{{ g_\omega (t_1)  }} -   e_{{ g_\omega (t_2)  }}     \|<\frac{\epsilon}{\|A\|}$. Using $h_A (g_\omega(t_2))=0$ we have:
\begin{eqnarray*}
|\ip{Ae_{g_\omega (t_1)}}{e_{g_\omega (t_2)}}|= |\ip{Ae_{g_\omega (t_1)} -  Ae_{g_\omega (t_2)} }{e_{g_\omega (t_2)}}|\leq \|A\|\|e_{g_\omega (t_1)}-e_{g_\omega (t_2)}\|\|e_{g_\omega (t_2)}\|<\epsilon
\end{eqnarray*}
 Then, using the Cauchy-Schwarz inequality
$$|\ip{Ae_{t_1}}{e_{t_2}}|\leq \left(\sum_{|\omega|=n}|\nu_\omega(t_1)|^2\right)^{\frac12}\left(\sum_{|\omega|=n}|\nu_\omega(t_2)|^2\epsilon^2\right)^{\frac12}=\epsilon.$$

Since $\epsilon$ was arbitrary, we get $\ip{Ae_{t_1}}{e_{t_2}}=0$ for all $t_1,t_2$, and since these vectors span the entire space, we get $A=0$ on $K$ and therefore $A=0$ on $H$.
\end{proof}

\begin{lemma}\label{lem4.9}
Let $F$ be a compact invariant set and let $h$ be a real valued continuous fixed point of the Ruelle operator. Then the sets
$$S:=\{x_0\in F : h(x_0)=\sup_{x\in F} h(x)\} \mbox{ and } I:=\{x_0\in F : h(x_0)=\inf_{x\in F}h(x)\}$$
are compact and invariant.
\end{lemma}

\begin{proof}
Let $x_0\in S$. We have that, if the transition $x_0\rightarrow g_i(x_0)$ is possible, then $g_i(x_0)\in F$ and $h(g_i(x_0))\leq \sup_{x\in F}h(x)= h(x_0)$. If the transition $x_0\rightarrow g_i(x_0)$ is not possible, then $\nu_i(x_0)=0$. Therefore
$$h(x_0)=\sum_{i=0}^{N-1}|\nu_i(x_0)|^2h(g_i(x_0))=\sum_{i,i\rightarrow g_i(x_0) \mbox{ possible}}|\nu_i(x_0)|^2h(g_i(x_0))$$$$\leq \sum_{i,i\rightarrow g_i(x_0) \mbox{ possible}}|\nu_i(x_0)|^2h(x_0)=h(x_0).$$
Thus, we must have equality in the inequality, which means that $h(g_i(x_0))=h(x_0)$ whenever $\nu_i(x_0)\neq 0$. Therefore $g_i(x_0)\in S$ if the the transition is possible, so $S$ is invariant. Similarly for $I$.
\end{proof}

\begin{lemma}\label{lem4.10}
If $h$ is a continuous fixed point of the Ruelle operator then $h$ is constant on minimal compact invariant sets.
\end{lemma}

\begin{proof}
If $Rh=h$, then the real and imaginary part of $h$ are also fixed points of the Ruelle operator, as $R$ is order preserving. Hence we may assume that $h$ is real-valued. Let $M$ be a minimal compact invariant set. By Lemma \ref{lem4.9}, the set
$$S=\{x_0\in M : h(x_0)=\sup_{x\in M}h(x)\}$$
is compact invariant and contained in $M$. Since $M$ is minimal, it follows that $S=M$, so $h$ is constant on $M$.
\end{proof}

\begin{lemma}\label{lem4.11}

Let $h$ be a continuous fixed point of the Ruelle operator. If $h=0$ on all minimal compact invariant sets, then $h=0$ on $\mathcal T$.

\end{lemma}

\begin{proof}
Taking the real and imaginary parts, we can assume $h$ is real valued. Let $x_0\in\mathcal T$. By Lemma \ref{lem4.9}, the set
$$S=\{ x\in \cj{\mathcal{O}(x_0)} : h(x)=\sup_{t\in\cj{\mathcal{O}(x_0)}}h(t)\}$$
is compact invariant, so, by Zorn's lemma, it contains a minimal compact invariant set $M_0$. Then, for all $x\in M_0$,
$$0=h(x)=\sup_{t\in\cj{\mathcal{O}(x_0)}}h(t).$$
Similarly
$$0=\inf_{t\in\cj{\mathcal{O}(x_0)}}h(t).$$

So $h$ is 0 on $\cj{\mathcal{O}(x_0)}$. Since $x_0$ was arbitrary, $h$ is 0 everywhere.

\end{proof}

\begin{lemma}\label{lem4.12}

Let $M$ be a minimal compact invariant set. Then $H(M)$ is irreducible for the representation $(S_i)_{i=0}^{N-1}$.

\end{lemma}

\begin{proof}
Suppose $L$ is a closed subspace of $H(M)$ which is invariant for the representation $(S_i)$. Then, let $P_L$ be the projection from $H$ to $L$, $P_K$ the projection from $H$ to $K$ and $P_M$ the projection from $H$ to $H(M)$. By Theorem \ref{dil} and Lemma \ref{las1}, the function
$$h_L(t)=\ip{P_KP_LP_Ke_t}{e_t}=\ip{P_Le_t}{e_t},\quad(t\in\mathcal T),$$
is a continuous fixed point of the Ruelle operator. Also since $0\leq P_L\leq P_M$, we have, with $h_M(t)=\ip{P_KP_MP_Ke_t}{e_t}=\ip{P_Me_t}{e_t}$,
$0\leq h_L\leq h_M.$
From Lemma \ref{lem4.6.1} we get that $h_M=0$ on all minimal compact invariant sets different than $M$. Indeed, if $M'$ is such a set then, for $t\in M'$, $P_Me_t=0$, so
$h_M(t)=\ip{P_Me_t}{e_t}=0$.

Since $0\leq h_L\leq h_M$, we get that $h_L$ is zero on all minimal compact invariant sets different than $M$. Also, from Lemma \ref{lem4.10}, $h_L$ is constant $c$ on $M$. Since $h_M=1$ on $M$, we get that $ch_M-h_L=0$ on all minimal compact invariant sets. Therefore, by Lemma \ref{lem4.11}, $h_L=ch_M$ on $\mathcal T$. This implies, by Lemma \ref{lem4.8}, that $P_L=cP_M$. Since $P_L$ and $P_M$ are projections, it follows that $c=0$ or $c=1$, and this means that $L$ is either $\{0\}$ or $H(M)$, so $H(M)$ is irreducible.
\end{proof}

\begin{lemma}\label{lem4.12.1}
If $M_1,\dots, M_p$ is a complete list, without repetitions, of the minimal compact invariant sets, then 
$$\oplus_{i=1}^p H(M_i)=H.$$

\end{lemma}

\begin{proof}
Let $H_0=\oplus_i H(M_i)$ and let $P_{H_0}$ be the corresponding projection. We know that $P_{H_0}$ commutes with the representation $(S_i)$ and we consider the fixed point of the Ruelle operator $h_{H_0}(t)=\ip{P_KP_{H_0}P_Ke_t}{e_t}=\ip{P_{H_0}e_t}{e_t}$, $t\in\mathcal T$. 

For every $i=1,\dots,p$ and every $t\in M_i$, since $e_t\in H(M_i)$, we have that $h_{H_0}(t)=\ip{P_{H_0}e_t}{e_t}=\ip{e_t}{e_t}=1$. Thus $1-h_{H_0}$ is 0 on every minimal compact invariant set. Thus, according to Lemma \ref{lem4.11}, $1-h_{H_0}$ is zero everywhere, and by Lemma \ref{lem4.8}, $I-P_{H_0}=0$, which means that $H_0=H$. 
\end{proof}

\begin{lemma}\label{lem4.13}
Let $M_1$, $M_2$ be two distinct minimal compact invariant sets. Then the restrictions of the representation $(S_i)_{i=0}^{N-1}$ to $H(M_1)$ and $H(M_2)$ are disjoint. 
\end{lemma}

\begin{proof}
Let $A:H(M_1)\rightarrow H(M_2)$ be an operator that intertwines the two representations. Extend $A$ from $H$ to $H$, by letting $A=0$ on the orthogonal complement of $H(M_1)$. Then $A$ commutes with the representation $(S_i)$. We claim that the fixed point of the Ruelle operator $h_A(t):=\ip{P_KAP_Ke_t}{e_t}=\ip{Ae_t}{e_t}=0$ for all $t\in\mathcal T$. 

If $t\in M_1$, then $e_t\in H(M_1)$ and $Ae_t\in H(M_2)$, so, by Lemma \ref{lem4.6.1}, we get that $h_A(t)=\ip{Ae_t}{e_t}=0$. 

If $t$ is in a minimal compact invariant set $M$ different than $M_1$, then $H(M)$ is orthogonal to $H(M_1)$, by Lemma \ref{lem4.6.1}, so $Ae_t=0$ by definition. Therefore $h_A(t)=\ip{Ae_t}{e_t}=0$ in this case too. So, $h_A$ is zero on all the minimal compact invariant sets, therefore $h_A$ is zero everywhere, according to Lemma \ref{lem4.11}, which implies that $A=0$, by Lemma \ref{lem4.8}. 

\end{proof}

\end{proof}

\begin{proof}[Proof of Theorem \ref{th2.11}]

Assume now that the maps $g_i$ are one-to-one and that the minimal compact invariant sets are finite. We prove a Lemma. 

\begin{lemma}\label{lem4.17}
Assume that all the maps $g_i$, $i=0,\dots,N-1$ are one-to-one. Let $M$ be a minimal finite invariant set. Then for any $t_1\neq t_2$ in $M$, $e_{t_1}\perp e_{t_2}$.

\end{lemma}

\begin{proof}
Let $d:=\min\{\dist(x,y) :x,y\in M, x\neq y\}$. Let $0<c<1$ be a common contraction ratio for the maps $(g_i)$ and let $n$ be large enough, so that $c^n \dist(t_1,t_2)<d$. Then, if $|\omega|\geq n$, and the transitions $t_1\rightarrow g_\omega(t_1)$ and $t_2\rightarrow g_\omega(t_2)$ are possible (in several steps), then $g_\omega(t_1),g_\omega(t_2)\in M$ and 
$$\dist(g_\omega(t_1),g_\omega(t_2))\leq c^n \dist (t_1,t_2)<d,$$
so $g_\omega(t_1)=g_\omega(t_2)$ and, since the maps $g_i$ are one-to-one, it follows that $t_1=t_2$, a contradiction. Thus, for $|\omega|\geq n$ one of the transitions  $t_1\rightarrow g_\omega(t_1)$ and $t_2\rightarrow g_\omega(t_2)$ is not possible. Then, using the Cuntz relations
$$\ip{e_{t_1}}{e_{t_2}}=\sum_{|\omega|=n}\ip{S_\omega^*e_{t_1}}{S_\omega^*e_{t_2}}=\sum_{|\omega|=n}\nu_\omega(t_1)\cj\nu_\omega(t_2)\ip{e_{g_\omega(t_1)}}{e_{g_\omega(t_2)}}=0.$$

\end{proof}

Lemma \ref{lem4.17} also shows that the random walk on $M$ is also separating. Since in a minimal finite invariant set $M$, the orbit of any point is $M$, it follows that the random walk on $M$ is irreducible. By Proposition \ref{pr2_14}, $\{V_i\}$ is simple on $K(M)$ for any minimal finite invariant set. 

We know that the Cuntz representation on $H(M_i)$ is irreducible, and therefore, with Theorem \ref{thm1}, we have that $\Span\{S_\omega e_{c_i} : \omega\in\Omega\}=H(M_i)$ and so $\lim_n P_{H_{c_i}^{(n)}}v=P_{H(M_i)}v$, for any $v\in H(M_i)$. 

The proof that we have a Parseval frame is very similar to the end proof of Theorem \ref{thm1}. 

Let $v\in K$. We have, for all $n\geq0$, using the fact that $\{V_j\}$ is reversing on all the spaces $K(M_i)$:

$$\sum_{i=1}^p\left\|P_{H_{c_i}^{(n)}}v \right\|^2=\sum_{i=1}^p\sum_{\omega\in\Omega_{c_i}^{(n)}}\left|\ip{S_\omega e_{c_i}}{v}\right|^2=\sum_{i=1}^p\sum_{\omega\in\Omega_{c_i}^{(n)}}\left|\ip{P_KS_\omega e_{c_i}}{v}\right|^2
=\sum_{i=1}^p\sum_{\omega\in\Omega_{c_i}^{(n)}}\left|\ip{V_\omega e_{c_i}}{v}\right|^2$$
$$=\sum_{i=1}^p\sum_{\omega\in\Omega_{c_i}^{(0)}}\sum_{\stackrel{\beta_1,\dots\beta_n}{\mbox{ cycle words for $c_i$}}}\left|\ip{V_{\omega\beta_1\dots\beta_n}e_{c_i}}{v}\right|^2$$$$
=\sum_{i=1}^p\sum_{\omega\in\Omega_{c_i}^{(0)}}\sum_{\stackrel{\beta_1,\dots\beta_n}{\mbox{ cycle words for $c_i$}}}|\nu_{\beta_1}(c_i)|^2\dots|\nu_{\beta_n}(c_i)|^2\left|\ip{V_{\omega_0}e_{c_i}}{v}\right|^2.$$

But
$$\sum_{\beta_i\mbox{ cycle word for $c_i$}}|\nu_{\beta_i}(c)|^2=1,$$
by Lemma \ref{lem4.15}, so we obtain
$$\sum_{i=1}^p\left\|P_{H_{c_i}^{(n)}}v \right\|^2=\sum_{i=1}^p\sum_{\omega\in\Omega_{c_i}^{(0)}}\left|\ip{V_{\omega_0}e_{c_i}}{v}\right|^2.$$
Then,
$$\left\| v\right\|^2=\lim_{n\rightarrow \infty}\sum_{i=1}^p\left\|P_{H_c^{(n)}}v \right\|^2=\sum_{i=1}^p\sum_{\omega\in\Omega_{c_i}^{(0)}}\left|\ip{V_{\omega_0}e_{c_i}}{v}\right|^2.$$

\end{proof}

The next proposition gives a sufficient condition for all the minimal compact invariant sets to be finite. 

\begin{proposition}\label{pr3.38}
Suppose the Assumptions \ref{as1} hold. Assume that at least one of the minimal compact invariant sets is finite. Let $X_0$ be the attractor of the iterated function system $\{g_i\}_{i=0}^{N-1}$. Assume that the zero sets $\{x\in X_0: \nu_i(x)=0\}$ are finite, and that all the maps $g_i$ are one-to-one, for all $i\in\{0,\dots,N-1\}$. Then all the minimal compact invariant sets are finite. 
\end{proposition}

\begin{proof}
Let $M_0$ be a finite minimal invariant set and let $M\neq M_0$ be a minimal compact invariant set. Let $x_0\in M_0$. Then, by Lemma \ref{lem4.11-0},  $\cj{\mathcal O(x_0)}=M_0$. Since $M_0$ is finite, we get that $\mathcal O(x_0)=M_0\ni x_0$, so, there exist $i_0,\dots,i_{p-1}$ such that the transition $x_0\rightarrow g_{i_{p-1}}\dots g_{i_0}(x_0)=x_0$ is possible, so $|\nu_\omega(x_0)|>0$, where $\omega=i_0\dots i_{p-1}$.

Since $\nu_\omega$ is continuous, there exists $\delta>0$ such that, if $d(x,x_0)<\delta$ then $|\nu_\omega(x)|>0$. 

Let $c$ be a common contraction ration for the maps $\{g_i\}$. Let $n\in\bn$ such that $c^{np}\textup{diam}(X_0)<\delta$. Let $x\in M$. Suppose the transition $x\rightarrow g_\omega^n(x)$ is possible. Then $g_\omega^n(x)\in M$. Also, 
$$d(g_\omega^n(x),x_0)=d(g_\omega^n(x),g_\omega^n(x_0))\leq c^{np}d(x,x_0)\leq c^{np}\textup{diam}(X_0)<\delta,$$
since $x\in M$, $x_0\in M_0$ and $M,M_0\subseteq X_0$. Therefore $|\nu_\omega(g_\omega^n(x))|>0$ and therefore the transition $g_\omega^n(x)\rightarrow g_\omega g_\omega^n(x)$ is possible. This means that $g_\omega^{n+1}(x)$ is in $M$. Also,
$$d(g_\omega^{n+1}(x),x_0)=d(g_\omega g_\omega^n(x),g_\omega(x_0)) \leq c^p d(g_\omega^n(x),x_0)<\delta.$$
By induction, $g_\omega^{n+k}(x)\in M$ and $d(g_\omega^{n+k}(x),x_0)<\delta$ for all $k\in\bn$. 

But, as $k\rightarrow\infty$, $g_\omega^{n+k}(x)$ converges to the fixed point of $g_\omega$, which is $x_0$. Since $g_\omega^{n+k}(x)\in M$ and $M$ is closed, it follows that $x_0\in M$, which contradicts the fact that $M\neq M_0$. 

Thus, the transition $x\mapsto g_\omega^n(x)$ is not possible so $\nu_{\stackrel{\omega\omega\dots \omega}{n\textup{ times}}}(x)=0$. Since $x\in M\subseteq X_0$, we have 
$$x\in \bigcup_{k=0}^{n-1}\bigcup_{j=0}^{p-1}(g_{i_j}\dots g_{i_0})^{-1}(\{y\in M : \nu_{i_{j+1}}(y)=0\}),$$
which is a finite set. Since $x$ was arbitrary in $M$, it follows that $M$ is finite. 
\end{proof}

\begin{corollary}\label{co3.39}
In the hypotheses of Theorem \ref{th3.5}, assume in addition that the operators $\{V_i\}$ are Cuntz isometries. If $\{V_i\}$ is reversing on $K$,  then for all $i=1,\dots,p$, $c\in M_i$, the numbers $|\nu_j(c)|$ are either 0 or 1, and $\{V_\omega e_{c_i} : \omega\in\Omega_{c_i}^{(0)},i=1,\dots,p-1\}$ is an orthonormal basis for $K$.
\end{corollary}

\begin{proof}
If $\{V_i\}$ is reversing on $K$, then for $c\in M_i$, if $\nu_j(c)\neq 0$, we have $V_je_{g_j(c)}=\cj\nu_j(c) e_c$. But, since $V_j$ is an isometry $\|V_j e_{g_j(c)}\|=1$, so $|\nu_j(c)|=1$. This implies that $\nu_k(c)=0$ for all $k\neq j$.

The fact that $\{V_\omega e_{c_i} : \omega\in\Omega_{c_i}^{(0)},i=1,\dots,p-1\}$ is an orthonormal basis for $K$ follows from Theorem \ref{th3.5} since the Cuntz dilation $\{S_j\}$ of $\{V_j\}$ is just $\{V_j\}$ on $K$.
\end{proof}

\section{Examples}\label{secex}

\begin{example}\label{ex4.1}
We start with a general example which combines exponential functions and piecewise constant functions on the attractor of an affine iterated function system. 

\begin{definition}\label{def4.1}
The Hilbert space will be the $L^2$-space associated to the invariant measure of an affine iterated function system.

Let $R$ be a $d\times d$ expansive integer matrix (i.e., all eigenvalues $\lambda$ have $\left| \lambda \right| > 1$). Let $B\subseteq \Z^{d}$, $0\in B$, $\left| B \right| = N$. Consider the affine iterated function system
\[
    \tau_b(x) = R^{-1}(x + b),\quad ( x\in \R^{d}, b\in B).
\]

Let $X_B$ be the attractor $X_B$ of the iterated function system $(\tau_b)_{b\in B}$, i.e., the unique compact set with the property that 
$$X_B=\bigcup_{b\in B}\tau_b(X_B).$$

Let $\mu_B$ be the invariant measure of the iterated function system, i.e., the unique Borel probability measure such that 

$$\int f\,d\mu_B=\frac{1}{N}\sum_{b\in B}\int f\circ\tau_b\,d\mu_B,$$
for all bounded Borel functions on $\br^d$. See \cite{Hut81} for details. 

We say that the measure $\mu_B$ has {\it no overlap} if 
$$\mu_B(\tau_b(X_B)\cap\tau_{b'}(X_B))=0,\mbox{ for all $b\neq b'$ in $B$.}$$

 Assume that $\mu_B$ has no overlap. Define $\mathcal{R}:X_B \to X_B$ by 
\[
    \mathcal{R}x = Rx - b, \quad x\in \tau_b(X_B),	
\]

so that $\mathcal R\tau_b(x)=x$ for all $x\in X_B$, $b\in B$. 

Now, to construct the co-isometry $(V_i)_{i=0}^{M-1}$ on the Hilbert space $L^2(\mu_B)$, suppose that there exist some points $l_0,\dots,l_{M-1}$ in $\Z^{d}$,  $l_0=0$ and $a_{i,b}\in \C$ ($i\in \{0,\dots,M-1\}, b\in B$) such that the matrix

\begin{align}\label{eq:0.1}
    \frac{1}{\sqrt[]{N}}\left( e^{2\pi i R^{-1} b \cdot l_i}a_{i,b} \right)_{i\in \{0,\dots,M-1\}, b\in B}
\end{align}
has orthonormal columns (so it is an isometry). In other words, for all $b,b'$ in $B$: 

\begin{equation}
\frac{1}{N}\sum_{i=0}^{M-1}a_{i,b}\cj a_{i,b'}e^{2\pi i R^{-1}(b-b')\cdot l_i}=\delta_{b,b'}.
\label{eq0.1.1}
\end{equation}

 Define the function $m_i$ on $X_B$ by: 
\begin{align*}
    m_i(x) = e^{2\pi i l_i\cdot x}\sum_{b\in B} a_{l_i,b}\chi_{\tau_b(X_B)}(x),\quad (x\in X_B,i\in\{0,\dots,M-1\}).
\end{align*} 
Here $\chi_A$ is the characteristic function of the subset $A$.

Define the operators $V_i$ on $L^2(\mu_B)$ by 
\begin{equation}
V_if(x) = m_l(x)f(\mathcal{R}x),\quad( x\in X_B, f\in L^2(\mu_B)).
\label{eq0.1.2}
\end{equation}

\end{definition}
\end{example}

\begin{proposition}\label{pr4.3}
    We have the following:

    \begin{enumerate}
        \item For $i\in\{0,\dots,M-1\}$ and $f\in L^2(\mu_B)$: 
				
				\begin{equation}
				V_i^{*}f(x) = \frac{1}{N}\sum_{b\in B}^{} \overline{m_i}(\tau_b(x))f(\tau_b(x)),\quad (x\in X_B). 
				\label{eqex1}
				\end{equation}
				
        In particular, for $e_t(x) = e^{2\pi i t\cdot x}$ ($t\in \R^{d})$,
            
               \begin{equation}
							 V_i^{*}e_t = \left( \frac{1}{N}\sum_{b\in B} e^{2\pi i(R^{\top})^{-1} (t - l_i)\cdot b} \overline{a_{i,b}} \right)\cdot e_{(R^{\top}) ^{-1}(t - l_i)}
							 \label{eqex2}
							 \end{equation}
            
        \item $\{V_i\}_{i=0}^{M-1}$ is a co-isometry on $L^2(\mu_B)$.

				This shows that the Assumptions \ref{as1} are satisfied, with $e(t)(x)=e^{2\pi it\cdot x}$, $t,x\in\br^d$ and 
				\begin{equation}
				\nu_i(t)=\frac{1}{N}\sum_{b\in B} e^{2\pi i(R^{\top})^{-1} (t - l_i)\cdot b} \overline{a_{i,b}},\quad g_i(x)=(R^{\top}) ^{-1}(t - l_i).
				\label{eqex3}
				\end{equation}

        \item If, in addition, the matrix in \eqref{eq:0.1} is unitary, then the operators $\{V_i\}$ are Cuntz isometries.
    \end{enumerate}
\end{proposition}
\begin{proof}
    We have that, for $f, g\in L^2(\mu)$,
    $$
        \left<V_i^{*}f,g \right> = \left<f, V_i g \right>$$
                                 $$= \int f(x)\overline{m_i}(x)\overline{g}(\mathcal{R}x)d \mu_B (x)$$
                                 $$= \frac{1}{N}\sum_{b\in B}^{} \int f(\tau_b x) \overline{m_i}(\tau_b x) \overline{g}(\mathcal{R}\tau_b x) d \mu_B (x)$$                                
																$$= \int_{}^{} \left( \frac{1}{N}\sum_{b\in B}^{} \overline{m_i}(\tau_b x)f(\tau_b x) \right) \overline{g}(x) d \mu_B(x).$$
 
    In particular, we have that
$$V_i^{*}e_t(x) = \frac{1}{N} \sum_{b\in B}^{} \overline{m_i}(\tau_{b}(x)) e_t(\tau_b(x))$$
                      $$= \frac{1}{N}\sum_{b\in B}^{} e^{-2\pi i l_i \cdot R^{-1}(x + b)} \overline{a_{i,b}}e^{2\pi i t \cdot R^{-1}(x + b)}$$
                      $$= \frac{1}{N}\sum_{b\in B}^{} e^{2\pi i (R^{\top})^{-1} (t - l_i)\cdot (x + b)}\overline{a_{i,b}}$$
                      $$=\left( \frac{1}{N}\sum_{b\in B} e^{2\pi i(R^{\top})^{-1} (t - l_i)\cdot b} \overline{a_{i,b}} \right)e^{2\pi i (R^{\top}) ^{-1}(t - l_i)\cdot x}$$
                      $$=\left( \frac{1}{N}\sum_{b\in B} e^{2\pi i(R^{\top})^{-1} (t - l_i)\cdot b} \overline{a_{i,b}} \right) e_{(R^{\top}) ^{-1}(t - l_i)}(x),$$
    which implies \eqref{eqex1}.

    Secondly, for $f\in L^2(\mu)$, a fixed $b_0\in B$, and $x\in \tau_{b_{0}}(X_B)$, we have
    $$(\sum_{i=0}^{M-1} V_i V_i^{*}f)(x) = \sum_{i=0}^{M-1} V_i \left(\frac{1}{N}\sum_{b\in B} \overline{m_i}(\tau_b(x))f(\tau_b(x))\right)$$
                                           $$= \sum_{i=0}^{M-1} \frac{1}{N}\sum_{b\in B}^{} m_i(x) \overline{m_i}(\tau_b \mathcal{R}(x)) f(\tau_b\mathcal{R}(x))$$
                                           $$= \frac{1}{N}\sum_{b\in B} \sum_{i=0}^{M-1}m_i(x)\overline{m_i}(\tau_b \mathcal{R}(x)) f(\tau_b\mathcal{R}(x))$$
                                           $$= \frac{1}{N}\sum_{b\in B} \sum_{i=0}^{M-1} m_i(x) e^{-2\pi i l_i\cdot \tau_b \mathcal{R} x}\overline{a_{i,b}} f(\tau_b \mathcal{R}(x))$$
                                           $$=  \frac{1}{N}\sum_{b\in B} \sum_{i=0}^{M-1} e^{2\pi i l_i\cdot x} e^{-2\pi i l_i\cdot (x + R^{-1}(b - b_{0}))}a_{i,b_{0}} \overline{a_{i,b}} f(\tau_b \mathcal{R}(x))$$
                                           $$= \sum_{b\in B} f(\tau_b \mathcal{R}(x)) \left( \frac{1}{N}\sum_{i=0}^{M-1} a_{i,b_{0}}\overline{a_{i,b}} e^{2\pi i l_i\cdot R^{-1}(b_{0} - b)}\right)$$
                                           $$= \sum_{b\in B} f(\tau_b \mathcal{R}(x)) \delta_{b,b_{0}} = f(x).$$
         If the matrix in \eqref{eq:0.1} is unitary, then we have that

$$V_{i'}^{*} V_{i} f(x) = V_{i'}\left( m_i(x)f(\mathcal{R}x)\right)
                                   = \frac{1}{N}\sum_{b\in B}\overline{m_{i'}}(\tau_b(x))m_i(\tau_b(x)) f(\mathcal{R}\tau_b(x))$$
                                   $$= \frac{1}{N}\sum_{b\in B}^{} e^{-2\pi i l_{i'}\cdot \tau_b(x)}\overline{a_{i',b}} e^{2\pi i l_i \cdot \tau_b(x)} a_{i,b} f(x)$$
                                   $$= f(x) \left( \frac{1}{N}\sum_{b\in B} e^{-2\pi i l_{i'} \cdot R^{-1}(x + b)} \overline{a_{i, b}}e^{2\pi il_i \cdot R^{-1}(x + b)} a_{i,b}\right)$$
                                   $$= f(x) \left( e^{2\pi i (R^{\top})^{-1}(l_i - l_{i'})\cdot x}\frac{1}{N}\sum_{b\in B}^{} e^{-2\pi i (R^\top)^{-1} l_{i'} \cdot b} \overline{a_{i', b}}e^{2\pi i(R^\top)^{-1}l_i \cdot b}a_{i,b} \right)$$
                                   $$= f(x)e^{2\pi i(R^\top)^{-1}(l_i - l_{i'})\cdot x  } \delta_{i,i'}= f(x) \delta_{i,i'}.$$

\end{proof}

\begin{example}\label{ex4.2}
In this example, we consider the weighted Fourier frames studied in \cite{PiWe17,DuRa16,DuRa18}.

As in Example \ref{ex4.1}, consider an affine iterated function system on $\br^d$, 
$$\tau_b(x)=R^{-1}(x+b),\quad (b\in B, x\in\br^d).$$
Assume now that there exist some points $l_0,\dots,l_{M-1}$ in $\bz^d$, $l_0=0$, and some complex numbers $\alpha_0,\dots,\alpha_{M-1}$, $\alpha_0=1$, such that the matrix 
\begin{equation}
\frac{1}{\sqrt[]{N}}\left( e^{2\pi i R^{-1} b \cdot l_i}\alpha_i \right)_{i\in \{0,\dots,M-1\}, b\in B}
\label{eq4.2.1}
\end{equation}
is an isometry, i.e., 
\begin{equation}
\frac{1}{N}\sum_{i=0}^{M-1}e^{2\pi i R^{-1}(b-b')\cdot l_i}|\alpha_i|^2=\delta_{b,b'},\quad (b,b'\in B).
\label{eq4.2.2}
\end{equation}

By \cite[Theorem 1.6]{DHL19}, we get that the measure $\mu_B$ has no overlap. Indeed, according to the cited reference, we just have to make sure that the elements in $B$ are not congruent modulo $R\bz^d$. But this follows, by contradiction, from \eqref{eq4.2.2}.

Note that this corresponds to a special case in Example \ref{ex4.1}, when, for all $i\in\{0,\dots, M-1\}$, we have $a_{i,b}=\alpha_i$ for all $b\in B$, that is $a_{i,b}$ is independent of $b$. Using Proposition \ref{pr4.3}, we obtain that the isometries $V_i$ are given by 
\begin{equation}
V_i f(x)=\alpha_i  e^{2\pi i l_i\cdot x} f(\mathcal Rx), \quad (f\in L^2(\mu_B)),
\label{eq4.2.3}
\end{equation}
and they satisfy the Assumptions \ref{as1}, with 
\begin{equation}
g_i(t)=(R^\top)^{-1}(t-l_i), \quad \nu_i(t)=\cj\alpha_i m_B(g_i(t)),\quad (t\in\br^d,i\in\{0,\dots,M-1\}),
\label{eq4.2.4}
\end{equation}
where 
\begin{equation}
m_B(t)=\frac{1}{N}\sum_{b\in B}e^{2\pi i b\cdot t},\quad (t\in \br^d).
\label{eq4.2.5}
\end{equation}

The set $\{0\}$ is a minimal invariant set, because $l_0=0$ and $\alpha_0=1$, so the only possible transition from $0$ is to $0\stackrel{0}{\rightarrow}0$ with probability $|\nu_0(0)|^2=1$.

In dimension $d=1$, since $m_B$ is a trigonometric polynomial, it has finitely many zeros in the attractor $X_L$ of the maps $\{g_i\}$, and therefore we can use Proposition \ref{pr3.38}, to conclude that all minimal invariant sets are finite.

We check that $\{V_j\}$ is reversing on every space $K(M_i)$ for all minimal invariant sets $M_i$, $i=1,\dots,p$. For this, we will use \cite[Proposition 4.2]{DuRa16}, which shows that, for every $t\in M_i$, $b\cdot t\in\bz$, for all $b\in B$. We include the statement of that result, because it gives a lot of information about the structure of the minimal finite invariant sets in this situation:

\begin{proposition}\label{prinv}\cite[Proposition 4.2]{DuRa16}
Assume $\alpha_i\neq 0$ for all $i\in \{0,\dots,M-1\}$. 
Let $\mathcal M $ be a non-trivial finite, minimal invariant set. Then, for every two points $t,t'\in \mathcal M $ the transition is possible from $t$ to $t'$ in several steps. In particular, every point in the set $\mathcal M $ is a cycle point. The set $\mathcal M $ is contained in the interval $\left[\frac{\min(-l_i)}{R-1},\frac{\max(-l_i)}{R-1}\right]$. 

If $t$ is in $\mathcal M $ and if there are two possible transitions $t\rightarrow g_{l_1}(t)$ and $t\rightarrow g_{l_2}(t)$, then $l_1\equiv l_2(\mod R)$.

Every point $t$ in $\mathcal M $ is an extreme cycle point, i.e., $|m_B(t)|=1$ and if $t\rightarrow g_{l_0}(t)$ is a possible transition in one step, then $\{i : (l_0-l_i)\cdot R^{-1}b\in \bz\mbox{ for all } b\in B\}=\{i: l_i\equiv l_0(\mod R)\}$  and
\begin{equation}
\sum_{i,l_i\equiv l_0(\mod R)}|\alpha_l|^2=1.
\label{eqac}
\end{equation}

In particular $t\cdot b\in\bz$ for all $b\in B$. 
\end{proposition}

Take $c\in M_i$, and $j\in\{0,\dots,M-1\}$ such that $\nu_j(c)\neq 0$. This means that the transition $c\rightarrow g_j(c)$ is possible and so $g_j(c)\in M_i$ and $b\cdot g_j(c)\in\bz$ for all $b\in B$. Then, for $x\in\tau_b(X_B)$, we have :
$$V_j e_{g_j(c)} (x)=\alpha_j e^{2\pi  i l_j x} e^{2\pi i g_j(c)(Rx-b)}=\alpha_je^{2\pi i(l_j x+ g_j(c)\cdot (Rx-b))}$$$$=\alpha_je^{2\pi i(l_j x+ g_j(c)\cdot Rx)}=\alpha_ie^{2\pi i(l_j x+ c x-l_jx)}=\alpha_j e^{2\pi i cx}.$$
Also 
$$\cj\nu_j(c)=\alpha_j \cj m_B(g_j(c))=\alpha_j \frac{1}{N}\sum_{b\in B}e^{-2\pi  ib g_j(c)}=\alpha_j.$$
So $V_j e_{g_j(c)}=\cj\nu_j(c) e_c$.

Thus, we can apply Theorem \ref{th3.5}. So, we pick a point $c_i$ in each minimal invariant set $M_i$. Recall that $\Omega_{c_i}^{(0)}$ is the set of all words in $\Omega$ that do not end in a cycle word for $c_i$. We compute $V_\omega e_{c_i}$, for $\omega=\omega_1\dots\omega_n\in \Omega$, and we show that 
\begin{equation}
V_\omega e_{c_i}=\alpha_{\omega_1}\dots\alpha_{\omega_n} e_{l_{\omega_1}+Rl_{\omega_2}+\dots R^{n-1}l_{\omega_n}+R^n c_i}.
\label{eq4.2.6}
\end{equation}
Indeed, using the fact that $bc_i\in\bz$ for all $b\in B$, and $R\in\bz$, take $x\in \tau_b(X_B)$, and we have:
$$V_{\omega_n}e_{c_i}(x)=\alpha_{\omega_n}e^{2\pi i(l_{\omega_n} x+c_i (Rx-b))}=\alpha_{\omega_n}e_{l_{\omega_n}+Rc_i}(x).$$
Then 
$$V_{\omega_{n-1}}V_{\omega_{n}}e_{c_i}=\alpha_{\omega_{n-1}}\alpha_{\omega_n}e_{l_{\omega_{n-1}}+Rl_{\omega_{n}}+R^2c_i}.$$
The relation \eqref{eq4.2.6} then follows by induction.

Thus, with Theorem \ref{th3.5} and Corollary \ref{co3.39}, we obtain 

\begin{corollary}\label{cor4.7}\cite[Theorem 1.6]{DuRa18}
In dimension $d=1$, let $M_i$, $i=1,\dots, p$ be all the minimal finite invariant sets, and pick $c_i\in M_i$ for each $i=1,\dots,p$. The family of weighted exponential functions 
$$\left\{ \alpha_{\omega_1}\dots\alpha_{\omega_n} e_{l_{\omega_1}+Rl_{\omega_2}+\dots+R^{n-1}l_{\omega_n}+R^n c_i} : \omega\in\Omega_{c_i}^{(0)}, i=1,\dots,p\right\}$$
is a Parseval frame for $L^2(\mu_B)$.
\end{corollary}

\begin{corollary}\label{cor4.8}\cite[Theorem 8.4]{DJ06}
In dimension $d=1$, suppose that the matrix 
$$\frac{1}{\sqrt{N}}\left(e^{2\pi iR^{-1}b\cdot l_i}\right)_{i\in\{0,\dots,N-1\},b\in B}$$
is unitary. Let $M_i$, $i=1,\dots, p$ be all the minimal invariant sets, and pick $c_i\in M_i$ for each $i=1,\dots,p$. The family of exponential functions 
$$\left\{  e_{l_{\omega_1}+Rl_{\omega_2}+\dots+R^{n-1}l_{\omega_n}+R^n c_i} : \omega\in\Omega_{c_i}^{(0)}, i=1,\dots,p\right\}$$
is an orthormal basis for $L^2(\mu_B)$.
\end{corollary}
\end{example}

\begin{example}\label{ex4.2.1}
In this example, we show that, in higher dimensions, it is possible to have minimal compact invariant sets which are infinite. Take
$$R:=\begin{bmatrix}
4&0\\1&4
\end{bmatrix},\quad B:=\left\{\begin{bmatrix} 0\\0\end{bmatrix},\begin{bmatrix} 0\\3\end{bmatrix},\begin{bmatrix} 1\\0\end{bmatrix},\begin{bmatrix} 1\\3\end{bmatrix}\right\}.$$
One can take 
$$L:=\left\{l_0=\begin{bmatrix} 0\\0\end{bmatrix},l_1=\begin{bmatrix} 2\\0\end{bmatrix},l_2=\begin{bmatrix} 0\\2\end{bmatrix},l_3=\begin{bmatrix} 2\\2\end{bmatrix}\right\},$$
and all $\alpha_i=1$, so that the matrix in \eqref{eq4.2.1} is unitary.

We have 
$$m_B(x,y)=\frac{1}{4}(1+e^{2\pi ix}+e^{2\pi i3y}+e^{2\pi  i(x+3y)})=\frac{1}{4}(1+e^{2\pi ix})(1+e^{2\pi i3y}).$$

We note that the set $S=\{(x,-2/3) : x\in\br\}$ is invariant. Indeed, $(R^{\top})^{-1}=\begin{bmatrix}
1/4&-1/16\\0&1/4
\end{bmatrix}$. 
So, if $l_i=(*,0)$ then the second component of $g_i(x,-2/3)$ is $-1/6$ and so 
\[
    \nu_i(x,-2/3)=m_B(g_i(x,-2/3))=m_B(*,-1/6)=0.
\]
If $l_i=(*,2)$, then the second component of $g_i(x,-2/3)$ is $-2/3$ so $g_i(x,-2/3)\in S$. Thus the only possible transitions from $(x,-2/3)$ are 
$$(x,-2/3)\stackrel{(0,2)^\top}{\rightarrow}(x/4+1/6,-2/3),\mbox{ and } (x,-2/3)\stackrel{(2,2)^\top}{\rightarrow}(x/4-1/3,-2/3).$$

Let $X_1$ be the attractor of the iterated function system $\sigma_0(x)=x/4+1/6$, $\sigma_2(x)=x/4-1/3$, i.e., the unique compact subset of $\mathbb R$ such that 
$$X_1=\sigma_0(X_1)\cup \sigma_2(X_1).$$

We claim that $X_1\times\{-2/3\}$ is a minimal compact invariant set. 

If we compute the fixed points of $\sigma_0$ and $\sigma_2$, which are $\frac29$ and $-\frac49$ respectively, the interval $[-\frac49,\frac29]$ is invariant for both $\sigma_0$ and $\sigma_2$, which implies that $X_1\subset [-\frac49,\frac29]$.

Note that, if $x\in X_1$, then 
$$m_B(x,-2/3)=\frac{1}{4}(1+e^{2\pi ix})=0,$$
if and only if $x=-\frac12$. Thus, the only points $(x,-2/3)\in X_1\times\{-2/3\}$ for which not both transitions are possible, could be $(\sigma_0^{-1}(-1/2),-2/3)=(-8/3,-2/3)$ and $(\sigma_1^{-1}(-1/2),2/3)=(-2/3,-2/3)$. But these points are outside the interval $[-\frac49,\frac29]\times\{-2/3\}$ so they are not in $X_1\times\{-2/3\}$. Therefore, for all points in $X_1\times \{-2/3\}$, both transitions are possible. But the closed orbit of any point in the attractor is the attractor itself, thus $X_1\times\{-2/3\}$ is minimal compact invariant. 
\end{example}

\begin{remark}\label{rem4.2.1}
A lot of information about the structure of the minimal invariant sets in higher dimensions can be found in \cite{CCR96}.
\end{remark}

\begin{example}\label{ex4.11}

We present here an example of a Cuntz representation, where a cycle point has two cycle words and therefore the family $\{S_\omega e_c : \omega\in\Omega_c^{(0)}\}$ is incomplete in the Hilbert space $H(\mathcal M)$ that corresponds to its minimal invariant set (by Proposition \ref{prinv}).

    We consider an affine iterated function system on $\br^2$, as in Definition \ref{def4.1}, determined by the scaling matrix ${\mathbf R}=\begin{bmatrix} 4&0\\0&2\end{bmatrix}$, and digits $B\times B'$ where $ B = \{0,2\}, B' = \{0,1\}$. Let $R=4, R'=2$,$ l_0=(0,0), l_1=(3,0),l_2= (4, 0),l_3 =(15, 0) $. We will order the set $B\times B' = \{(0,0), (2,0), (0,1), (2,1)\} $ for the sake of indexing the matrices to follow. Take 
\[
    \frac{1}{\sqrt[]{4}}(a_{i, (b,b')})_{i=0,\dots,3 , (b,b')\in B\times B'}
    = \frac{1}{2}\begin{pmatrix} 1 & 1 & 1 & 1 \\ 1 & 1 & \xi & \xi \\ 1 & 1 & -1 & -1 \\ 1 & 1 & -\xi & -\xi \end{pmatrix}, \quad |\xi| = 1; \, \xi \neq 1,-1
\] (this appears also in \cite{PiWe17}).  
\[
    \frac{1}{\sqrt[]{4}}\left( e^{2\pi i (\frac{1}{4}l_i^1 b + \frac{1}{2}l^2 b')}a_{i, (b,b')} \right)_{i=0,\dots,3, (b,b')\in B\times B'} = \frac{1}{2}\begin{pmatrix} 1 & 1 & 1 & 1 \\ 1 & -1 & \xi & -\xi \\ 1 & 1 & -1 & -1 \\ 1 & -1 & -\xi & \xi \end{pmatrix} 
\] is unitary where $l_i^1, l_i^2$ denote the coordinates of $l_i$. By an abuse of notation, we will often identify $(x,0)$ with $x$, to simplify the notation.
Define, for $i=0,\dots,3$,
\[
    m_i(x,x') = e^{2\pi i l_i \cdot (x,x')}\sum_{(b,b')\in B\times B' }^{} a_{i, (b,b')} \chi_{\tau_{b,b'}(X_B, X_B')}  (x,x'),
\]
and $S_i f (x,x') = m_l(x,x')f(\mathcal{R}x, \mathcal{R}'x')$.

Then from Proposition \ref{pr4.3}, since the matrix above is unitary, the operators $(S_i)$ form a representation of the Cuntz algebra $\mathcal{O}_4$. If we denote $e_{t,t'}(x,x') = e^{2\pi i (tx + t'x')}$, then we have that the $S_l^{*}$ have the form:
$$
S_i^{*} e_{t,t'} = \nu_i(t,t') e_{g_i(t,t')}, \text{ where}
$$
\begin{align*}
    \nu_i(t,t') &= \frac{1}{4}\sum_{(b,b')\in B\times B'}^{}e^{2\pi i (\frac{1}{4}(t-l_i)b + \frac{1}{2}t'b')}\overline{a_{i,(b,b')}} \\
    g_i(t,t')   &= \left( \frac{1}{4}(t-l_i), \frac{1}{2}t' \right)
.\end{align*}

We now find the compact minimal invariant sets. Assume $\mathcal M$ is a compact minimal invariant subset of $\mathcal{T} = \R^2$. Then if $(t,t')\in \mathcal M$, then there must exist a possible transition $(t,t')\to g_{i_1}(t,t') \to \ldots \to g_{i_1\ldots i_n}(t,t')\in \mathcal M$ for all $n$. By compactness of $\mathcal M$, there exists a convergent subsequence $g_{i_{1}i_{2}\ldots i_{n_k}}(t,t') \to (x_0, 0)\in \mathcal M$ as $k\to \infty$, since the second component of  $g_{i_{1}i_{2}\ldots i_{n_k}}(t,t')$ is $\frac{1}{2^{n_k}}t'$, which converges to $0$ as $k\to \infty$. Since $\mathcal M$ is minimal $\mathcal M = \overline{\mathcal{O}(x_0,0)}\subseteq \R\times \{0\}$. Thus $\mathcal M = M_1 \times \{0\} $, for some $M_1\subseteq \R$ (note that any transition from a point in $\R\times\{0\}$ leads to $\R\times\{0\}$). Now we calculate:
\begin{align*}
    \nu_i(t,0) = \frac{1}{4}\sum_{b\in B}^{} e^{2\pi i (\frac{1}{4}(t-l_i)b)}\sum_{b'\in B'}^{} \overline{a_{i,(b,b')}}
.\end{align*}
However, we see that $a_{i,(b,b')}$ is independent of $b$, so we denote
\begin{align*}
    \alpha_i &= \frac{1}{2}\sum_{b'\in B'}^{} a_{i,(b,b')} = \frac{1}{2}(a_{i, (0,0)} + a_{i,(0,1)}) = \frac{1}{2}(a_{i, (2,0)} + a_{i, (2,1)}) \\
    m_B(x)   &= \frac{1}{2}\sum_{b\in B}^{} e^{2\pi i bx}=\frac12(1+e^{2\pi i 2x})
.\end{align*}
Therefore we may write
\begin{align*}
    \nu_i(t,0) &= \overline{\alpha_i} m_B(\frac{1}{4}(t - l_i))\\
    g_i(t,0)   &= \left( \frac{1}{4}(t-l_i), 0 \right)
.\end{align*}

Thus, $M_1$ must be invariant for the maps $g_i(\cdot, 0)$ and the weights $\overline{\alpha_i}m_B(\frac{1}{4}(\cdot - l_i))$. Note that in particular, since $\alpha_2 =\frac12(1+(-1))= 0$, and $\alpha_i \neq 0$ for $i\neq 2$, it suffices to see that $M_1$ is invariant for the $g_i(\cdot , 0) $ and $m_B(\frac{1}{4}(\cdot -l_i)), i\in \{0,1,3\}$. We will show that $M_1 = \overline{\mathcal{O}(-1)} = \{-1, -4\}$ or $M_1 = \overline{\mathcal{O}(0)} = \{ 0\} $. 

With Proposition \ref{prinv}, we know that, for any $c\in M_1$, we must have $c\cdot b\in\bz$ for all $b\in B$. So $c\in\frac12\bz$. Also, $c\in [\frac{-15}{3},\frac{0}{3}]=[-5,0]$. Also, if $c\in M_1$ and the transition to $g_{i}(c)$ is possible, then $g_i(c)$ is in $M_1$ so it must be of the same form $\frac12\bz$. Note that $m_B(x)=0$ only if $x$ is of the form $\frac{2k+1}{4}$ for some $k\in\bz$. Thus, if $c=\frac{2j+1}{2}$ for some $j\in\bz$ then the transition $c\rightarrow g_0(c)=\frac{2j+1}{8}$ is possible and $g_0(c)$ is not in $\frac12\bz$. This means that $c$ cannot be in $M_1$. So we only have to check $\{-5,-4,-3,-2,-1,0\}$. $\{0\}$ is the trivial invariant set. Also $\{-4,-1\}$ is invariant. We have the possible transitions $-5\rightarrow g_1(-5)=-2\rightarrow g_0(-2)=-\frac12$ and $-3\rightarrow g_1(-3)=-\frac32$. Since $-\frac12$ and $-\frac32$ are not in $M_1$ it follows that neither are $-5,-3$ nor $-2$. 

In the case of $M_1 = \mathcal{O}(-1) = \{-4, -1\} $, we see that $-1$ has two distinct cycle words, $1$ and $3,0$. Indeed $-1\rightarrow g_1(-1)=\frac{-1-3}{4}=-1$ and $-1\rightarrow g_3(-1)=\frac{-1-15}{4}=-4\rightarrow g_0(-4)=\frac{-4}{4}=0$. Thus by Proposition \ref{princ}, we know that $\overline{\text{span}}\{S_\omega e_{(-1,0)}: \omega\in \Omega_{(-1,0)}^{(0)}\} \neq H(M_1)$.

\end{example}

\begin{example}\label{ex4.3}
 We use Theorem \ref{th3.5} to provide a class of Parseval frames as in \cite{DuRa20}, Theorem 3.11. Let $A$ be a $M\times N$ matrix such that $\frac{1}{\sqrt{N}}A^{*}A=I_N$ (hence $N\leq M$) and the first row is constant $\alpha_{0,j}=1$, $j=0,\dots,N-1$. With $\mathcal R(x)=Nx\text{ mod } 1$ and $k\in\{0,1,\dots M-1\}$ define
 $$m_k(x):=\sum_{j=0}^{N-1}a_{kj}\chi_{_{[j/N, (j+1)/N)}}(x) $$
 $$V_k:L^2[0,1]\to L^2[0,1], \quad V_kf(x):=m_k(x)f(\mathcal R(x)).$$

Note that this corresponds to the Example \ref{ex4.1} when $R=N$, $B=\{0,1,\dots,N-1\}$, $l_i=0$ for all $i=0,\dots,M-1$ and $a_{i,b}$ are the entries of the matrix $A$. Indeed, the attractor of the iterated function system $(\tau_b)_{b\in B}$ is $[0,1]$ and the invariant measure $\mu_B$ is the Lebesgue measure on $[0,1]$.

By Proposition \ref{pr4.3}, the Assumptions \ref{as1} are satisfied, with 
$$g_i(t)=\frac tN,\quad \nu_i(t)=\frac{1}{N}\sum_{j=0}^{N-1}e^{2\pi i\frac {j\cdot t}N}\cj a_{i,j}.$$

We will show that the only compact minimal invariant set is $\mathcal{M}=\{0\}$ then apply Theorem \ref{th3.5}.

The set $M=\{0\}$ is invariant because the only possible transition from $0$ is $0\stackrel{0}{\rightarrow}0$ with probability $|\nu_0(0)|^2=1$. 

 To show this is the only compact, minimal invariant set suppose by contradiction there is a compact minimal invariant set $\mathcal{N}$ with some $t\in \mathcal{N}$, $t\neq 0$. We argue that necessarily $0\in\mathcal{N}$ thus $\{0\}\subset \mathcal{N}$, contradicting the minimality of $\mathcal{N}$. If $t\in\mathcal{N}$ then $t/N\in\mathcal{N}$ because there must be at least one possible transition $t\to g_k(t)=t/N$ since $\sum_k |\nu_k(t)|^2=1$. Continuing in this fashion we get $a_n:=(t/N)^n\in \mathcal{N}$ for all $n\in\mathbb{N}$. By compactness and invariance $0=\lim_{n\to\infty}a_n \in\mathcal{N}$, as desired.  

Since the only cycle word for $0$ is $0$, with Theorem \ref{th3.5}, we obtain:

\begin{corollary}\cite[Theorem 1.4]{DuRa20}
The family of functions 
$$\{V_\omega {\bf 1}: \omega\in\Omega\mbox{ not ending in $0$}\}$$
is a Parseval frame for $L^2[0,1]$.
\end{corollary}

%To prove the fixed point uniqueness property for the associated Ruelle operator we only need $\norm{Av}^2=\norm{v}^2$
%with $v=(e^{-2 \pi i t j/M})_{j=0}^{M-1}$. Hence the family $\{V_{\omega}(1)\}_{\omega}$ is a Parseval frame in $L^2[0,1]$.

\end{example}

\begin{acknowledgements}
We would like to thank professor Deguang Han for very helpful conversations and to the anonymous referee for a very careful and thorough review of our paper.
\end{acknowledgements}

\bibliographystyle{alpha}	
\bibliography{eframes}

\end{document}